\documentclass[aop]{imsart}

\RequirePackage{amsthm,amsmath,amsfonts,amssymb}
\RequirePackage[numbers]{natbib}

\usepackage{pstricks} 

\usepackage[all,arc]{xy}
\usepackage{enumerate}
\usepackage{mathrsfs}

\usepackage{amsmath}
\usepackage{amsthm}
\usepackage{amsfonts}
\usepackage{amssymb}
\usepackage{amsbsy}
\usepackage{graphicx}
\usepackage{enumitem}
\usepackage{physics}
\usepackage{hyperref}
\usepackage{bbm}
\usepackage{longtable}
\usepackage{comment}  
\usepackage{float}
\usepackage{caption}

\newcommand{\mc}[1]{\mathcal{#1}}

\newcommand{\floor}[1]{ \lfloor #1 \rfloor }

\newcommand{\sse} {\subseteq}

\newcommand{\Z}{\mathbb{Z}}
\newcommand{\R}{\mathbb{R}}

\newcommand{\C}{\mathbb{C}}
\newcommand{\E}{\mathbb{E}}

\newcommand{\ra}{\rightarrow}
\newcommand{\toinf}{\ra \infty}
\newcommand{\ptl}{\partial}

\newcommand{\beq}{\begin{equation}}
\newcommand{\eeq}{\end{equation}}

\newtheorem{theorem}{Theorem}
\newtheorem{prop}[theorem]{Proposition}
\newtheorem{lemma}[theorem]{Lemma}
\newtheorem{cor}[theorem]{Corollary}
\theoremstyle{plain}

\newtheorem{remark}[theorem]{Remark}

\newcommand{\p}{\mathbb{P}}

\numberwithin{equation}{section}
\numberwithin{theorem}{section}

\newcommand{\largebox}{\Lambda}
\newcommand{\lbox}{\largebox}
\newcommand{\vertices}{\lbox_0}
\newcommand{\edges}{{\lbox_1}}
\newcommand{\plaquettes}{{\lbox_2}}
\newcommand{\wloop}{\gamma}
\newcommand{\supp}{\mathrm{supp}}
\newcommand{\mbbm}[1]{\mathbbm{#1}}
\newcommand{\ind}{\mbbm{1}}
\newcommand{\oneskel}{S_1(\lbox)}

\newcommand{\Hom}{\mathrm{Hom}}
\newcommand{\homsym}{\psi}

\newcommand{\groupid}{1}
\newcommand{\upath}{w}

\newcommand{\vortex}{V}
\newcommand{\scale}{s}

\newcommand{\homspace}{\Omega}
\newcommand{\swap}{U}

\newcommand{\calC}{\mathcal{C}}
\newcommand{\td}{\text{d}}

\begin{document}

\title{Correlation decay for finite lattice gauge theories at weak coupling}
\runtitle{Exponential decay of correlations}


\author[A]{\fnms{Arka}~\snm{Adhikari}
\ead[label=e1]{arkaa@stanford.edu}}
\address[A]{Department of Mathematics, Stanford University, 450 Jane Stanford Way, Building 380, Stanford, CA 94305\printead[presep={,\ }]{e1}}

\author[B]{\fnms{Sky}~\snm{Cao}
\ead[label=e2]{skycao@mit.edu}}
\address[B]{Department of Mathematics, Massachusetts Institute of Technology, 77 Massachusetts Ave, Cambridge, MA 02139\printead[presep={,\ }]{e2}}


\begin{abstract}
In the setting of lattice gauge theories with finite (possibly non-Abelian) gauge groups at weak coupling, we prove exponential decay of correlations for a wide class of gauge invariant functions,  which in particular includes arbitrary functions of Wilson loop observables.
\end{abstract}



\maketitle

\newtheorem{defn}[theorem]{Definition}

\section{Introduction}

Lattice gauge theories are statistical mechanical models which arise as discretizations of Yang--Mills theories. {They were systematically introduced by Wilson \cite{W1974} in 1974, and since then} they have been the subject of much study by both physicists and mathematicians alike. { The main physical motivation for the study of lattice gauge theories is to understand the Standard Model of physics, which describes the behavior of elementary particles at short distances. Mathematically, the goal is to construct a continuum Yang--Mills theory, and one of the main approaches towards this is by taking a continuum limit of lattice gauge theories. So far, this has not been achieved, but nonetheless there have been many seminal works proving various aspects about the behavior of lattice gauge theories. For instance, Ba\l aban \cite{balaban83}-\cite{balaban89b} proved ultraviolet stability (a notion of tightness) for lattice gauge theories in 3D and 4D via a renormalization group approach. { The monograph of Glimm and Jaffe goes into significant detail on constructive field theory, including the rigorous construction of $U(1)$ theory in 2-dimensions \cite{GlJa}. } For results on quark confinement, see the works \cite{borgs88, Ch2021, frohlichspencer82, gopfertmack82, guth80}. The works \cite{BS16, Ch15, CJ16, J16} look at large $N$ limits of lattice gauge theories and resulting forms of gauge-string duality. We emphasize that we have left out many important references in this paragraph; see \cite{Ch2018} for a more complete list as well as a historical overview.}

{ Associated to lattice gauge theories is a parameter $G$, which is typically taken to be a Lie group. However, in this paper, we will make a mathematical simplification and take $G$ to be a finite group. Of course, this lessens the direct relevance to physics, but on the other hand, lattice gauge theories with finite groups $G$ have been previously studied in the physics literature; see e.g. \cite{AF1984, CJR1979, Fro1979, MP1979, Seiler1982, Tom1993, WEG1971} for an incomplete list. Additionally, from a statistical mechanical point of view, such models are interesting because they give examples of discrete models with non-Abelian symmetries. These symmetries require the use of topology to handle, and result in relations to knot theory as pointed out by \cite{SV1989}. Such considerations do not appear in the usual discrete spin models such as the Ising or Potts models.}

This paper expands on the analysis of the works \cite{Adh2021, Cao2020} to answer a natural question in this area. The purpose of those papers was to calculate (for finite gauge groups) the expectations of certain natural observables associated to lattice gauge theories. Having computed expectations of observables, a natural next step is to show decay of correlations of those same observables. This is precisely the purpose of the present paper. In general, computing observable expectations and showing correlation decay are two related problems which however require different proof techniques, and thus the arguments of \cite{Adh2021, Cao2020} which give the former do not immediately give the latter. To prove correlation decay in this paper, we combine the technical tools developed in the previous papers with some crucial additional proof ideas that are tuned to the problem at hand.

{ The problem of showing exponential decay of correlations for a 4D lattice gauge theory (with gauge group a non-Abelian Lie group, say $\mathrm{SU}(3)$) is one of the central questions of lattice gauge theory, as it is intimately related to showing the existence of a mass gap. Following the discussion of \cite[Section 5]{Ch2018}, let $f_{\beta}(x)$ denote the correlation between an appropriately defined Wilson loop variable centered at $0$ and another centered at $x$. Then, the existence of a mass gap for a given $\beta \geq 0$ is equivalent to the following statement: there exists some $\xi(\beta) \in (0, \infty)$ such that
\[
     -\frac{1}{\xi(\beta)}=\lim_{|x| \to \infty} \frac{\log f_{\beta}(x)}{|x|}. 
\]
The conjecture is that the above holds for all $\beta \geq 0$ (and we furthermore expect that $\lim_{\beta \toinf} \xi(\beta) = \infty$). Given this discussion, the main result of the present work can be interpreted as showing the existence of a mass gap in the setting of finite (non-Abelian) gauge groups at large $\beta$.} { In general, it is more difficult to prove exponential decay of correlations at large $\beta$ compared to small $\beta$, as the latter case can usually be handled by a routine high-temperature expansion, and this works not just for finite groups, but also non-Abelian Lie groups (see \cite{os78}).}

\subsection{Previous work}

There has been much recent interest in computing the expectations of Wilson loop observables -- Chatterjee \cite{Ch2018} considered the case $G = \Z_2$, Forrstr\"{o}m et al. \cite{FLV2020} handled finite Abelian $G$, and Cao \cite{Cao2020} covered finite (possibly non-Abelian) $G$. There is also recent work by Garban and Sep\'{u}lveda \cite{GS2021} for $G = \mathrm{U}(1)$. For lattice gauge theories with an additional Higgs field (i.e., lattice Higgs models), Forsstr\"{o}m et al.~\cite{FLV2021} considered the finite Abelian case, and Adhikari \cite{Adh2021} considered the finite (non-Abelian) case. Forsstr\"{o}m \cite{Forsstrom2021b} has also analyzed a more relevant class of observables for lattice Higgs models. 

As for previous work on the decay of correlations, there is the classic monograph by Seiler \cite{Seiler1982}, which shows exponential decay of correlations for finite Abelian lattice gauge theories in a variety of settings, using cluster expansion techniques\footnote{It is claimed in the monograph that the given cluster expansions extend without difficulty from finite Abelian groups to general finite groups. However, as pointed out by Borgs \cite[Section 7]{Borgs1984}, this is not the case.}.~There is also recent work by Forsstr\"{o}m \cite{Forsstrom2021a}, which  also proves exponential decay in the Abelian case, using a certain probabilistic swapping argument, which relates correlations to percolation probabilities of the union of two independently sampled configurations. The general idea for this type of argument previously appeared before in the literature in other settings, see e.g. \cite{Aizenman82}. We will also use this general principle, however there are significant difficulties that arise in the non-Abelian case. We will discuss these difficulties, as well as difficulties in extending the cluster expansion of \cite{Seiler1982} to our setting in Section \ref{section:abelian-vs-nonabelian}.

Finally, for some recent work on decay of correlations for other statistical mechanical models, see \cite{AP2019, AHP2020, DX2021, DPSS2017, LT2021}.

\subsection{The difference between finite Abelian and finite non-Abelian}\label{section:abelian-vs-nonabelian}

In this section, we comment on the main differences between Abelian lattice gauge theories and non-Abelian lattice gauge theories (here all groups are assumed to be finite). Actually, at a probabilistic level, it is generally the case that there is no difference, in that whatever probabilistic statement that is true in the finite Abelian case should also have an analog in the finite non-Abelian case. However, this always has to be proven, and the proofs are not merely technical improvements on the arguments which work in the Abelian case. Rather, significant new ideas must be developed to handle the non-Abelian case. 

We now try to convey the reason why. The first part of the discussion will not be new to those who are familiar with Pirogov-Sinai theory. Recall that we are always working at low temperature (i.e. large $\beta$). As is usually the case for finite spin systems, at low temperatures it is best to work in terms of ``defects". E.g., for the Ising model, the defects are the edges with spin disagreements. The benefit is that now there is a unique ground state, which is the state with no defects (in the Ising model, this would be induced by either the all $+$ configuration or all $-$ configuration). At low temperature, the system can be analyzed as a small perturbation of this unique ground state. Mathematically, this is usually achieved by writing the partition function as a convergent cluster expansion. Once this is done, many (if not all) desired properties of the model can then be read off from general cluster expansion results (see Seiler's monograph \cite{Seiler1982}).

In order to write the partition function as a cluster expansion, the key property that one needs is that the partition function obeys a certain factorization property, which in words very roughly amounts to saying that disjoint defects appear independently of each other. Now for finite Abelian theories, this factorization property can be proven directly (see e.g. \cite[Lemma 3.2.3]{Cao2020}). However, for finite non-Abelian theories, the same proof does not work, and indeed cannot possibly work, because the exact same factorization property is simply not true anymore. To try to indicate why, consider an idealized situation with two defects as pictured:

\begin{figure}[H]
    \centering
    \captionsetup{width=\linewidth}
    \includegraphics[scale=0.4]{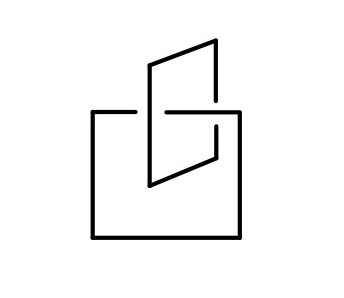} 
    \vspace{-5mm}
    \caption{}
    \label{fig:linked-loops}
\end{figure}

Here, the two defects are given by the two loops, and observe that they are linked. In the Abelian case, the aforementioned factorization property holds for these two defects, whereas in the non-Abelian case, factorization does not hold, {\it even though the two defects may be arbitrarily far apart from each other} (see \cite{SV1989}). Thus in order to obtain some sort of weaker factorization property, we must treat the above picture as a single ``connected" defect, even if there really are two connected components in the usual sense. This basic example points to the fact that non-Abelian theories lead to nontrivial topological considerations, and this seems to be a genuinely different phenomenon which is unseen elsewhere, as we are not aware of any other statistical mechanical models which lead to such considerations about the defects.

More generally, connected defects in non-Abelian theories can be thought of as elementary links (in the sense of knot theory), although they are referred to as knots in \cite{Cao2020, SV1989}. With this notion, a weaker form of the factorization property may be proved (this itself is nontrivial and requires algebraic topology -- see \cite[Section 4]{Cao2020}). However, even after arriving at the right notion of connected defect in non-Abelian theories, there are additional difficulties that prevent one from writing the partition function as a cluster expansion in this case.\footnote{In \cite{SV1989} a cluster expansion is claimed, however much of the proof is omitted, and we are unable to reconstruct the proof from the arguments that are given.} Perhaps the main difficulty is that the elementary defects now interact in a multi-body interaction, instead of just a two-body interaction (and a cluster expansion typically requires the latter). This essentially comes from the fact that one can have three knots which are linked yet pairwise unlinked (the classic example being the Borromean rings). 

Thus in summary, non-Abelian theories lead to a much more complicated notion of connected defect, and do not (as of yet) admit a cluster expansion. This is why the previous results \cite{Seiler1982, Forsstrom2021a} on correlation decay only hold for Abelian groups -- the former work relies on standard cluster expansion results, while the latter paper gives a more probabilistic argument that does not use cluster expansion, but which in the end still crucially relies on the decomposition into connected defects which (as we explained) is more straightforward in the Abelian case. 

By contrast, in this paper, we must delicately combine a probabilistic argument with a careful handling of the topological effects that appear. This is due to the fact that the elementary defects are (1) now much more complicated to define (as we tried to convey) and (2) having defined them, they are very delicate to work with. In particular, we need to have a precise understanding of the associated topological difficulties in order to apply any sort of probabilistic ``swapping" argument. These difficulties have no analogue in the Abelian case. In the end, we are forced to find novel arguments in order to combine the various intricate aspects that are present to arrive at our proof. For these reasons, we believe that our arguments for the general non-Abelian case are substantial improvements over previous proofs which work only in the Abelian case.

\subsection{Definitions and Notation}

We proceed to give the basic definitions and then state our main theorems. We start with a preliminary discussion of concepts and definitions necessary to understand the paper; one can refer to \cite{Cao2020} or \cite{Adh2021} for more details.

Given an integer $n \geq 1$, define $[n] := \{1, \ldots, n\}$. Let $G$ \label{notation:G} be a finite group, with the identity denoted by $\groupid$\label{notation:group-id}. We will commonly refer to $G$ as the gauge group. Let $\rho$\label{notation:rho} be a unitary representation of $G$, with dimension $d$, and let $\chi = \Tr \rho$\label{notation:chi} be the character of $\rho$. From here on, fix a finite lattice 
\[ \lbox\label{notation:Lambda} := ([a_1, b_1] \times \cdots \times [a_4, b_4])  \cap \Z^4, \]
where $b_i - a_i$ is the same for all $i \in [4]$. The results of this paper apply to any such $\lbox$. Let $\vertices$ be the set of vertices of $\lbox$. Let $\edges$ be the set of (nearest-neighbor) edges of $\lbox$. We implicitly assume that each edge $e = (x, y) \in \edges$ carries a positive orientation, i.e., $y = x + e_i$ for some $i \in [4]$. In this paper, we only work in dimension four, as it is the most relevant dimension for lattice gauge theories (see \cite{Ch2018}). However, we expect that our results can be extended to general dimensions, with some additional technical arguments.

We will refer to the elements $\sigma \in G^\edges$ as ``edge configurations" (because they assign edges to group elements). Given an edge configuration $\sigma \in G^\edges$ and an edge $e = (x, y) \in \edges$, we can naturally extend $\sigma$ to the negatively oriented version of $e$ by setting $\sigma_{(y, x)} := \sigma_e^{-1}$.

By a ``plaquette" $p$\label{notation:plaquette} in $\lbox$, we mean a unit square whose four boundary edges are in $\lbox$. Let $\plaquettes$ be the set of plaquettes in $\lbox$. For $p \in \plaquettes$, suppose the vertices of $p$ are $x_1, x_2, x_3, x_4$, in (say) counter-clockwise order. In an abuse of notation, for $\sigma \in G^\edges$, define
\beq\label{eq:sigma-p-def} \sigma_p := \sigma_{(x_1, x_2)} \sigma_{(x_2, x_3)} \sigma_{(x_3, x_4)} \sigma_{(x_4, x_1)}. \eeq
Define
\beq\label{eq:S-def} S_\lbox(\sigma) := \sum_{p \in \plaquettes} \Re (\chi(1) - \chi(\sigma_p)). \eeq
(By the conjugacy invariance of $\chi$, it does not matter which vertex of $p$ we choose to start at when defining $\sigma_p$.) For $\beta \geq 0$\label{notation:beta}, let $\mu_{\lbox, \beta}$\label{notation:mu-lbox-beta} be the probability measure on $G^\edges$ defined by
\beq\label{eq:def-lgt} \mu_{\lbox, \beta}(\sigma) := Z_{\lbox, \beta}^{-1} ~e^{-\beta S_\lbox(\sigma)}, \eeq
where $Z_{\lbox, \beta}$ is the normalizing constant. We say that $\mu_{\lbox, \beta}$ is the lattice gauge theory with gauge group $G$, on $\lbox$, with inverse coupling constant $\beta$. In this paper, we will work in the large $\beta$ regime, which is also known as the weak coupling regime.


Let $\wloop$ be a closed loop in $\lbox$, denoted by its sequence of oriented edges $e_1, \ldots, e_n$. We write $|\wloop|$ for the length of $\wloop$, i.e., the number of edges in $\wloop$. For $\sigma \in G^\edges$, define (in another abuse of notation)
\[ \sigma_\gamma := \sigma_{e_1} \cdots \sigma_{e_n}.\]
We say that $\sigma_\gamma$ is the holonomy of $\sigma$ around $\wloop$. Let $\chi_0$ be a character of $G$ (there is no relation between $\chi_0$ and the character $\chi$ which appears in the definition \eqref{eq:S-def} of $S_\lbox$). The Wilson loop observable $W_{\wloop, \chi_0}$ associated to $\wloop, \chi_0$ is defined as a function  $G^\edges \ra \C$ by the formula
\[ W_{\wloop, \chi_0}(\sigma) := \chi_0(\sigma_\wloop), ~~ \sigma \in G^\edges.\]
Wilson loop observables are the main observables of interest in lattice gauge theories.~See \cite[Section 4]{Ch2018} for the physical motivation in defining these observables, as well as for further discussion. Next, define
\[ \Delta_G \label{notation:Delta-G} := \min_{\substack{g \in G \\ g \neq \groupid}} \Re (\chi(\groupid) - \chi(g)). \]
We can almost state our main result. In the following, by the notation ``$\Sigma \sim \mu_{\lbox, \beta}$", we mean that $\Sigma$ is a $G^\edges$-valued random variable with distribution $\mu_{\lbox, \beta}$. We say that a function $f : G^k \ra \C$ is conjugacy invariant if, for any $g_1, \ldots, g_k \in G$ and $h_1, \ldots, h_k \in G$, we have that $f(h_1^{-1} g_1 h_1, \ldots, h_k^{-1} g_k h_k) = f(g_1, \ldots, g_k)$. A rectangle $B$ is a subset of $\lbox$ of the form $[x_1, y_1] \times \cdots \times[x_4, y_4]$, where we allow $x_i = y_i$. We let $|B|$ be the number of plaquettes contained in $B$.


\begin{theorem}\label{thm:main-general}
Let $\beta \geq \frac{1}{\Delta_G}(114 + 4 \log |G|)$. Let $L \geq 0$. Let $B_1, B_2 \sse \lbox$ be rectangles that are at a $\ell^\infty$ distance at least $L$ from each other (i.e., the $\ell^\infty$ distance between any vertex $x$ of $B_1$ and any vertex $y$ of $B_2$ is at least $L$). Let $k_1, k_2 \geq 1$, and let $f_1 : G^{k_1} \ra \C, f_2 : G^{k_2} \ra \C$ be conjugacy invariant functions. For $i = 1, 2$, let $\wloop_1^{(i)}, \ldots, \wloop_{k_i}^{(i)}$ be closed loops contained in $B_i$. Let $\Sigma \sim \mu_{\lbox, \beta}$. Then
\[\begin{split}
|\mathrm{Cov}(f_1(\Sigma_{\gamma_1^{(1)}}, \ldots, &\Sigma_{\gamma_{k_1}^{(1)}}), f_2(\Sigma_{\gamma^{(2)}_1}, \ldots, \Sigma_{\gamma^{(2)}_{k_2}}))| \leq \\
& 4 (4 \cdot 10^{24} |G|^2)^{|B_1| + |B_2|} \|f_1\|_{\infty} \|f_2\|_{\infty} e^{-(\beta / 2) \Delta_G (L-1)} .
\end{split}\]
\end{theorem}

\begin{remark}
Examples of $f_1, f_2$ that one could take in the theorem are Wilson loop observables, or, more generally, arbitrary functions of arbitrary numbers of Wilson loop observables.
\end{remark}

We now sketch the underlying idea of the proof, which is quite general and, in principle, could be applied to other types of spin systems or statistical mechanical models. For illustration purposes, suppose that we have some probability measure $\mu$ on the space $\Omega = \{\pm 1\}^{\vertices}$ of $\pm 1$ spin configurations. Let $\mu^{\otimes 2}$ be the two-fold product of $\mu$, so that $\mu^{\otimes 2}$ is a probability measure on $\Omega^2$. Let $f_1, f_2 : \Omega \ra \R$ be functions. The key assumption is: suppose that there is a bijection $T : \Omega^2 \ra \Omega^2$ such that for all $(\sigma_1, \sigma_2) \in \Omega^2$, we have that $\mu^{\otimes 2}(T(\sigma_1, \sigma_2)) = \mu^{\otimes 2}(\sigma_1, \sigma_2)$ and $f_1(\sigma_1) f_2(\sigma_2) = f_1(T_1(\sigma_1, \sigma_2)) f_2(T_1(\sigma_1, \sigma_2))$, where here $T_1 : \Omega^2 \ra \Omega$ is the first component function of $T$. Think of $T$ as a map that swaps $\sigma_1$ and $\sigma_2$ on some set of vertices (e.g., the support of $f_2$), which somehow satisfies the first condition (i.e., probabilities are preserved).

Now, let $\Sigma_1, \Sigma_2 \stackrel{i.i.d.}{\sim} \mu$. Observe that the first condition on $T$ implies that $T(\Sigma_1, \Sigma_2) \stackrel{d}{=} (\Sigma_1, \Sigma_2)$; indeed, that condition says that $\mu^{\otimes 2} \circ T = \mu^{\otimes 2}$, and since $T$ is a bijection, we obtain $\mu^{\otimes 2} \circ T^{-1} = \mu^{\otimes 2}$, and now note that $\mu^{\otimes 2} \circ T^{-1}$ is the law of $T(\Sigma_1, \Sigma_2)$ (and of course, $\mu^{\otimes 2}$ is the law of $(\Sigma_1, \Sigma_2)$). Thus in particular, $T_1(\Sigma_1, \Sigma_2) \stackrel{d}{=} \Sigma_1$. Combining the second condition on $T$ with this observation, we obtain
\[ \E[f_1(\Sigma_1) f_2(\Sigma_2)] = \E[f_1(T_1(\Sigma_1, \Sigma_2)) f_2(T_1(\Sigma_1, \Sigma_2))] = \E[f_1(\Sigma_1) f_2(\Sigma_1)], \]
from which we obtain $\mathrm{Cov}(f_1(\Sigma_1), f_2(\Sigma_1)) = 0$ (recall that $\Sigma_1, \Sigma_2 \stackrel{i.i.d.}{\sim} \mu$).

Of course, in practice, we will not have the bijection $T$ defined on all of $\Omega^2$. Rather, we will need to find a subset $E \sse \Omega^2$ on which we can define the bijection $T : E \ra E$ which satisfies the two listed properties. A modification of the previous argument then gives an upper bound on $\mathrm{Cov}(f_1(\Sigma_1), f_2(\Sigma_1))$ in terms of $\p((\Sigma_1, \Sigma_2) \notin E)$ (so we want to take the set $E$ as large as possible). 

Section \ref{section:swapping-argument} gives a more precise statement of this general argument. Needless to say, the main difficulty will be in actually constructing $T$ and $E$ for given $f_1, f_2$, and then bounding $\p((\Sigma_1, \Sigma_2) \notin E)$. For this, we will rely on some concepts introduced in \cite{Cao2020}, which we review in Sections \ref{section:preliminaries} and \ref{section:knot-decomposition}. Finally, if the reader is interested, in Appendices \ref{appendix:higgs-large-kappa} and \ref{appendix:higgs-small-kappa}, we show how to handle Higgs models at large and small $\kappa$ (these models were recently analyzed in \cite{Adh2021}). The corresponding analysis for models with  a Higgs' field is more technical and requires more notation that is not easy to unify with the existing arguments. Thus for the convenience of the reader, we have placed the discussion of these ideas in the appendices.

\subsection{Acknowledgements}

We thank Sourav Chatterjee for suggesting this problem, as well as for helpful conversations. A.A. would also like to thank NSF award 2102842 for its support.

\section{Preliminaries}\label{section:preliminaries}


In this section (following \cite{Cao2020}), we rephrase everything in the language of algebraic topology, using the starting observation that edge configurations $\sigma \in G^\edges$ can be thought of as homomorphisms of the fundamental group of the lattice. The main benefit of this rephrasing is that it allows us to prove Lemma \ref{lemma:vortices-well-separated-gauge-fix-decomp} (this was essentially already done in \cite{Cao2020}, see the proof of \cite[Lemma 4.2.21]{Cao2020}). This lemma, in turn, allows us to prove Lemma \ref{lemma:bijection-existence} and Corollary \ref{cor:bijection-existence}, which form the main technical foundation for the proof of Theorem \ref{thm:main-general}.

Recall that a cell complex is a certain type of topological space obtained by assembling cells of varying dimensions; see, e.g., Section 0.2.4 of \cite{STILL1993}. In our case, the cells will be unit squares of dimension at most two, i.e., vertices, edges, and plaquettes. So for us, a one dimensional cell complex, or 1-complex, is a space consisting of vertices and edges, and thus it is a graph. A two dimensional cell complex, or 2-complex, is a space consisting of vertices, edges, and plaquettes.

In what follows, if we define a 1-complex by specifying a collection of edges, then that 1-complex is understood to also include the vertices of the edges in the collection. Similarly, if we define a 2-complex by specifying a collection of plaquettes, then that 2-complex is understood to also include the vertices and edges of the plaquettes in the collection.

Let $\oneskel$\label{notation:oneskel} denote the 1-skeleton of $\lbox$, i.e., the 1-complex obtained from the edges of $\lbox$. Throughout this paper, fix some vertex $x_0 \in \vertices$. Recall that the fundamental group $\pi_1(\oneskel, x_0)$ is a group of equivalence classes of closed loops starting and ending at $x_0$. The equivalence relation is given by setting equivalent any two loops of the form $\ell_1 e e^{-1} \ell_2$ and $\ell_1 \ell_2$, where $e = (x, y)$ is an edge of $\oneskel$, $\ell_1$ is a path in $\oneskel$ from $x_0$ to $x$, and $\ell_2$ is a path in $\oneskel$ from $x$ to $x_0$. The group operation is induced by loop concatenation.


Next, we observe that edge configurations $\sigma \in G^\edges$ naturally induce a homomorphism from $\pi_1(\oneskel, x_0)$ to $G$. 

\begin{defn}\label{def:psi-x0}
Let $\homspace := \Hom(\pi_1(\oneskel, x_0), G)$ be the set of homomorphisms from $\pi_1(\oneskel, x_0)$ to $G$. Define the map $\psi_{x_0} : G^\edges \ra \homspace$ as follows. Let $\sigma \in G^\edges$. Given an element $\xi \in \pi_1(\oneskel, x_0)$, suppose that $\xi$ can be represented as a loop that traverses the edges $e_1, \ldots, e_n$. Define $(\homsym_{x_0}(\sigma))(\xi) := \sigma_{e_1} \cdots \sigma_{e_n}$. Note that this is well-defined, since if $ee^{-1}$ appears, then this gives the term $\sigma_e \sigma_{e^{-1}} = \sigma_e \sigma_e^{-1} = \groupid$.
\end{defn}

\begin{remark}
The map $\psi_{x_0}$ is exactly the same as the map $\psi_T^{x_0}$ defined in \cite[Section 4.1]{Cao2020}. The difference in notation is due to the fact that in \cite[Section 4.1]{Cao2020}, a spanning tree $T$ of $\oneskel$ is fixed, and $\psi_T^{x_0}$ is defined in terms of this spanning tree. It was not noted in that paper, but it turns out that the definition of $\psi_T^{x_0}$ is independent of $T$, and thus we prefer to write $\psi_{x_0}$ in the present paper. 
\end{remark}


An immediate consequence of Definition \ref{def:psi-x0} is the following lemma, whose proof is omitted.

\begin{lemma}\label{lemma:psi-conjugacy-preserve}
Let $\sigma \in G^\edges$. For any loop $\wloop$ in $\lbox$, and any path $\ell$ in $\lbox$ from $x_0$ to the initial vertex of $\wloop$, we have that $\sigma_\wloop$ is conjugate to $(\psi_{x_0}(\sigma))(\ell^{-1}\wloop \ell)$.
\end{lemma}

The following lemma is essentially \cite[Lemma 4.1.1]{Cao2020}. Due to the differences in notation between that paper and the present paper, we provide a proof of this lemma in Appendix \ref{appendix:edge-configurations-and-homomorphisms}.

\begin{lemma}\label{lemma:psi-many-to-one}
For any $\psi \in \homspace$, there are exactly $|G|^{|\vertices| - 1}$ edge configurations $\sigma \in G^\edges$ such that $\homsym_{x_0}(\sigma) = \psi$.
\end{lemma}

\begin{remark}
This lemma shows that homomorphisms $\psi \in \homspace$ can be thought of as gauge equivalence classes.
\end{remark}

Next, in the following sequence of definitions and lemmas culminating in Lemma \ref{lemma:psi-Sigma-distribution} below, we proceed to reinterpret the lattice gauge theory \eqref{eq:def-lgt} as giving a random homomorphism. 

\begin{defn}\label{def:xi-gamma}
For all $x \in \vertices$, fix an arbitrary path $\ell(x_0, x)$ from $x_0$ to $x$ in $\oneskel$. Let $\wloop$ be a loop in $\lbox$ with the initial vertex $x$. Define $\xi_\wloop := \ell(x_0, x) \wloop \ell(x_0, x)^{-1} \in \pi_1(\oneskel, x_0)$. For a plaquette $p \in \plaquettes$, let $\gamma_p$ be the same loop which traverses the boundary of $p$ that was used in the definition \eqref{eq:sigma-p-def} of $\sigma_p$. Define $\xi_p := \xi_{\gamma_p}$. 
\end{defn}

\begin{defn}
For $\psi \in \homspace$, define
\[ \supp(\psi) := \{p \in \plaquettes : \psi(\xi_p) \neq \groupid\}. \]
This definition does not depend on the particular choices of $\ell(x_0, x), x \in \vertices$ from Definition \ref{def:xi-gamma}, as a consequence of Lemma \ref{lemma:psi-conjugacy-preserve} and the fact that $\psi_{x_0} : G^\edges \ra \Omega$ is onto (which itself follows by Lemma \ref{lemma:psi-many-to-one}).
\end{defn}

The following fact was proven in \cite[Lemma 4.3.7]{Cao2020}, and thus we omit the proof here.

\begin{lemma}\label{lemma:homomorphism-number-bound}
Let $P \sse \plaquettes$. The number of $\psi \in \homspace$ such that $\supp(\psi) = P$ is at most $|G|^{|P|}$.
\end{lemma}

\begin{defn}
For $\beta \geq 0$, define $\varphi_\beta : G \ra (0, 1]$ by
\[\varphi_\beta(g) := \exp\big(-\beta\Re (\chi(1) - \chi(g))\big), ~~ g \in G. \]
Recalling the definition \eqref{eq:S-def} of $S_\lbox$, we have that
$e^{-S_\lbox(\sigma)} = \prod_{p \in  \plaquettes} \varphi_\beta(\sigma_p)$.
\end{defn}

\begin{defn}\label{def:nu}
Define the probability measure $\nu_{\lbox, \beta}$ on $\homspace$ as follows:
\[ \nu_{\lbox, \beta}(\psi) := (Z^\nu_{\lbox, \beta})^{-1} \prod_{p \in \plaquettes} \varphi_\beta(\psi(\xi_p)), \]
where $Z^\nu_{\lbox, \beta}$ is the normalizing constant.
\end{defn}

\begin{lemma}\label{lemma:psi-Sigma-distribution}
Let $\Sigma \sim \mu_{\lbox, \beta}$. Then $\psi_{x_0}(\Sigma) \sim \nu_{\lbox, \beta}$.
\end{lemma}
\begin{proof}
For any $\psi \in \homspace$, we have that
\[ \p(\psi_{x_0}(\Sigma) = \psi) = Z_{\lbox, \beta}^{-1} \sum_{\substack{\sigma \in G^\edges \\ \psi_{x_0}(\sigma) = \psi}} \prod_{p \in \plaquettes} \varphi_\beta(\sigma_p). \]
By Lemma \ref{lemma:psi-conjugacy-preserve}, for any $\sigma \in G^\edges$ such that $\psi_{x_0}(\sigma) = \psi$, we have that $\varphi_\beta(\sigma_p) = \varphi_\beta(\psi(\xi_p))$ for all $p \in \plaquettes$. Combining this with Lemma \ref{lemma:psi-many-to-one}, we obtain 
\[ \p(\psi_{x_0}(\Sigma) = \psi) \propto |G|^{|\vertices| - 1} \prod_{p \in \plaquettes} \varphi_\beta(\psi(\xi_p)) \propto \prod_{p \in \plaquettes} \varphi_\beta(\psi(\xi_p)). \]
The desired result now follows.
\end{proof}

\begin{lemma}\label{lemma:switch-to-random-homomorphism}
Let $k \geq 1$, and let $f := G^k \ra \C$ be a conjugacy invariant  function. Let $\gamma_1, \ldots, \gamma_k$ be closed loops in $\lbox$. Let $\Sigma \sim \mu_{\lbox, \beta}$, and let $\Psi \sim \nu_{\lbox, \beta}$. Then
\[ \E f(\Sigma_{\gamma_1}, \ldots, \Sigma_{\gamma_k}) = \E f(\Psi(\xi_{\gamma_1}), \ldots, \Psi(\xi_{\gamma_k})).\]
\end{lemma}
\begin{proof}
By Lemma \ref{lemma:psi-conjugacy-preserve} and the assumption that $f$ is conjugacy invariant, we have that 
\[ f(\Sigma_{\wloop_1}, \ldots, \Sigma_{\wloop_k}) = f\big((\psi_{x_0}(\Sigma))(\xi_{\gamma_1}), \ldots, (\psi_{x_0}(\Sigma))(\xi_{\gamma_k})\big). \]
By Lemma \ref{lemma:psi-Sigma-distribution}, we have that $\psi_{x_0}(\Sigma) \sim \nu_{\lbox, \beta}$. The desired result follows.
\end{proof}

We close this section by noting that $\psi_{x_0}$ is a bijection when restricted to certain subsets of $G^\edges$.

\begin{defn}
Given a spanning tree $T$ of $\oneskel$, define
\[ GF(T) := \{\sigma \in G^\edges : \sigma_e = \groupid \text{ for all $e \in \edges$}.\}.\]
Here ``$GF$" stands for ``gauge-fixed".
\end{defn}

The following lemma is essentially \cite[Lemma 4.1.6]{Cao2020}. Due to the differences in notation between that paper and the present paper, we provide a proof of this lemma in Appendix \ref{appendix:edge-configurations-and-homomorphisms}.

\begin{lemma}\label{lemma:gauge-fixed-bijection-homomorphism}
For any spanning tree $T$ of $\oneskel$, $\psi_{x_0} : GF(T) \ra \homspace$ is a bijection. Moreover, $\supp(\sigma) = \supp(\psi_{x_0}(\sigma))$ for all $\sigma \in GF(T)$. For any loop $\wloop$ in $\lbox$, and any path $\ell$ in $\lbox$ from $x_0$ to the initial vertex of $\wloop$, we have that $\sigma_\wloop$ is conjugate to $(\psi_{x_0}(\sigma))(\ell^{-1} \wloop \ell)$.
\end{lemma}




\section{A general swapping argument}\label{section:swapping-argument}

In this section, we give the general ``swapping map" idea behind our proof of decay of correlations, as previously sketched after the statement of Theorem \ref{thm:main-general}.

\begin{defn}[Swapping map]\label{def:swapping-map}
Let $E \sse \homspace^2$. Let $h_1, h_2 : \homspace \ra \C$ be arbitrary functions. Let $T : E \ra E$. We say that $T$ is a swapping map with respect to $(E, h_1, h_2, \nu_{\lbox, \beta})$ if $T$ is a bijection, and if additionally the following hold for all $(\psi_1, \psi_2) \in E$:
\begin{enumerate}
    \item Let $\nu_{\lbox, \beta}^{\otimes 2}$ be the two-fold product of $\nu_{\lbox, \beta}$ on $\homspace^2$. Then $\nu_{\lbox, \beta}^{\otimes 2}(T(\psi_1, \psi_2)) = \nu_{\lbox, \beta}^{\otimes 2}((\psi_1, \psi_2))$.
    \item Let $(\tilde{\psi}_1, \tilde{\psi}_2) = T(\psi_1, \psi_2)$. Then
    \[ h_1(\psi_1) h_2(\psi_2) =  h_1(\tilde{\psi}_1) h_2(\tilde{\psi}_1). \]
\end{enumerate}
\end{defn}


The following lemma shows that the existence of a swapping map leads to a covariance bound.

\begin{lemma}\label{lemma:swapping-gives-covariance-bound}
Let $E \sse \homspace^2$. Let $h_1, h_2 : \homspace \ra \C$ be arbitrary functions. Let $T : E \ra E$ be a swapping map with respect to $(E, h_1, h_2, \nu_{\lbox, \beta})$. Let $\Psi_1, \Psi_2 \stackrel{i.i.d.}{\sim} \nu_{\lbox, \beta}$. Then
\[ |\mathrm{Cov}(h_1(\Psi_1), h_2(\Psi_1))| \leq 2 \|h_1\|_\infty \|h_2\|_\infty \p((\Psi_1, \Psi_2) \notin E). \]
\end{lemma}
\begin{proof}
Let $F := \{(\Psi_1, \Psi_2) \in E\}$. We claim that
\beq\label{eq:swapping-claim} \E (h_1(\Psi_1) h_2(\Psi_1) \ind_F) = \E (h_1(\Psi_1) h_2(\Psi_2) \ind_F). \eeq
Given this claim, we obtain
\[\begin{split}
\mathrm{Cov}(h_1(\Psi_1), h_2(\Psi_1)) &= \E (h_1(\Psi_1) h_2(\Psi_1)) - \E (h_1(\Psi_1) h_2(\Psi_2)) \\
&= \E (h_1(\Psi_1) h_2(\Psi_1) \ind_{F^c}) - \E (h_1(\Psi_1) h_2(\Psi_2) \ind_{F^c}).
\end{split}\]
We may then trivially bound both
\[ |\E (h_1(\Psi_1) h_2(\Psi_1) \ind_{F^c})|, |\E (h_1(\Psi_1) h_2(\Psi_2) \ind_{F^c})| \leq \|h_1\|_\infty \|h_2\|_\infty \p(F^c), \]
from which the desired result follows.

Thus it just remains to show the claim \eqref{eq:swapping-claim}. We have that
\[ \E (h_1(\Psi_1) h_2(\Psi_1) \ind_F) = \sum_{(\psi_1, \psi_2) \in E} \nu_{\lbox, \beta}^{\otimes 2} ((\psi_1, \psi_2)) h_1(\psi_1) h_2(\psi_1). \]
Since $T : E \ra E$ is a bijection, the right hand side above is equal to
\[ \sum_{(\psi_1, \psi_2) \in E} \nu_{\lbox, \beta}^{\otimes 2} (T(\psi_1, \psi_2)) h_1(T_1(\psi_1, \psi_2)) h_2(T_1(\psi_1, \psi_2)), \]
where $T_1 : E \ra \homspace$ is the first coordinate function of $T$. Using conditions (1) and (2) of Definition \ref{def:swapping-map}, we obtain that the above is further equal to
\[ \sum_{(\psi_1, \psi_2) \in E} \nu_{\lbox, \beta}^{\otimes 2} ((\psi_1, \psi_2)) h_1(\psi_1) h_2(\psi_2) = \E (h_1(\Psi_1) h_2(\Psi_2) \ind_F), \]
as desired.
\end{proof}

The following proposition shows that swapping maps exist for functions of the form specified in Theorem \ref{thm:main-general}. 

\begin{prop}\label{prop:existence-of-swapping-map}
Let $\beta, L, B_1, B_2$ be as in Theorem \ref{thm:main-general}. There exists a set $E(B_1, B_2) \sse \homspace^2$ such that for $\Psi_1, \Psi_2 \stackrel{i.i.d.}{\sim} \nu_{\lbox, \beta}$, we have
\[ \p((\Psi_1, \Psi_2) \notin E(B_1,  B_2)) \leq 2 (4 \cdot 10^{24} |G|^2)^{|B_1| + |B_2|} e^{-(\beta / 2) \Delta_G (L-1)}. \]
Moreover, there exists a bijection $T : E(B_1,  B_2) \ra E(B_1,  B_2)$ such that $T$ is a swapping map with respect to $(E(B_1,  B_2), h_1, h_2, \nu_{\lbox, \beta})$, for any functions $h_1, h_2 : \homspace \ra \C$ of the following form. Let $k_1, k_2 \geq 1$, and let $f_i : G^{k_i} \ra \C$ be a conjugacy invariant function for $i = 1, 2$. For $i = 1, 2$, let $\wloop^{(i)}_1, \ldots, \wloop^{(i)}_{k_i}$ be closed loops contained in $B_i$, and then for $\psi \in \homspace$, let $h_i(\psi) := f_i(\psi(\xi_{\gamma^{(i)}_1}), \ldots, \psi(\xi_{\gamma^{(i)}_{k_i}}))$.
\end{prop}

We next note that Theorem \ref{thm:main-general} follows directly from Proposition \ref{prop:existence-of-swapping-map}.

\begin{proof}[Proof of Theorem \ref{thm:main-general}]
This is a direct consequence of Lemmas \ref{lemma:switch-to-random-homomorphism} and \ref{lemma:swapping-gives-covariance-bound}, and Proposition \ref{prop:existence-of-swapping-map}.
\end{proof}

The remainder of the paper is devoted to proving Proposition \ref{prop:existence-of-swapping-map}. In Section \ref{section:knot-decomposition}, we introduce the technical tools that are essential to the ensuing arguments. Then in Section \ref{section:construction-of-swapping-map}, we construct the event $E(B_1,  B_2)$ and the map $T$, and finally in Section \ref{section:probability-bound}, we bound the probability $\p((\Psi_1, \Psi_2) \notin E(B_1,  B_2))$.

\section{Knot decomposition}\label{section:knot-decomposition}

In this section, we review the concept of ``knot decomposition" as well as some other results which were introduced in \cite{Cao2020}. These results will allow us to prove Lemma \ref{lemma:bijection-existence}, which (along with Corollary \ref{cor:bijection-existence}, its immediate corollary) is the main result of this section. Recall the discussion of 2-complexes at the beginning of Section \ref{section:preliminaries}. We start with a series of definitions, most which have previously appeared in \cite[Section 4]{Cao2020}.

\begin{defn}
Given a rectangle $B$ contained in $\lbox$, let $S_2(B)$\label{notation:S-2-B} denote the 2-complex obtained by including all plaquettes of $B$. Let $\partial S_2(B)$\label{notation:partial-S-2-B} be the 2-complex obtained by including all plaquettes which are on the boundary of $B$, but not on the boundary of $\lbox$. Note if $B$ is contained in the interior of $\lbox$, then $\partial S_2(B)$ is simply the 2-complex made of all boundary plaquettes of $B$. Let $S_2^c(B)$\label{notation:S-2-c-B} be the 2-complex obtained by including all plaquettes of $\lbox$ that are not in $B$, as well as all plaquettes in $\partial S_2(B)$.
\end{defn}

\begin{defn}[Well separated]\label{def:well-separated}
Given plaquette sets $P_1, P_2 \sse \plaquettes$, we say that $P_1, P_2$ are well separated, or $P_1$ is well separated from $P_2$, if there exists a rectangle $B$ in $\lbox$ such that $P_1 \sse S_2(B)$, $P_2 \sse S_2^c(B)$, and no plaquettes of $P_1$ or $P_2$ are contained in $\partial S_2(B)$. For such a $B$, we say that $P_1, P_2$ are well separated by $B$, or that $B$ well separates $P_1, P_2$. Note this definition is not symmetric in $P_1, P_2$.
\end{defn}

Recall that a topological space is said to be simply connected if it is path connected and has trivial fundamental group. 

\begin{defn}[Good rectangle]\label{def:good-rectangle}
Let $B$ be a rectangle in $\lbox$. We say that $B$ is good if $S_2(B)$, $S_2^c(B)$, and $\ptl S_2(B)$ are all simply connected.
\end{defn}

\begin{remark}
The motivation for the preceding definition is to have Lemma \ref{lemma:vortices-well-separated-gauge-fix-decomp} (see also \cite[Lemmas 4.1.9 and 4.2.21]{Cao2020}). Actually, this lemma should probably hold for general rectangles, at the cost of an additional argument that we prefer not to give (we are trying to keep the topological arguments to a minimum). Thus, we will work exclusively with good rectangles in this paper.
\end{remark}

The following lemma gives the main examples of good rectangles that will be relevant for us.

\begin{lemma}\label{lemma:good-rectangle-examples}
Let $B$ be a rectangle in $\lbox$ such that either: (a) all side lengths of $B$ are strictly less than the side length of $\lbox$, or (b) the vertices of $B$ is a set of the form $\{x \in \vertices : x_i \leq k\}$ or $\{x \in \vertices : x_i \geq k\}$ for some $i \in [4]$, $k \in \Z$ (i.e., $B$ is the intersection of $\lbox$ and a half space which is parallel to one of the coordinate axes). Then $B$ is a good rectangle.
\end{lemma}

Using the previously introduced notions, we can define a partition of any plaquette set $P \sse \plaquettes$, as follows.

\begin{defn}[Knot decomposition]\label{def:knot-decomposition}
For every plaquette set $P \sse \plaquettes$, we fix a maximal partition
\[ P = K_1 \cup \cdots \cup K_m,\]
such that for all $1 \leq i \leq m - 1$, we have that $K_i$ is well separated from $K_{i+1} \cup \cdots \cup K_m$ by a good rectangle $B_i$ in $\lbox$. Here, ``maximal" means that for all $1 \leq i \leq m$, there does not exist a further partition $K_i = K \cup K'$ such that $K$ is well separated from $K'$ by a good rectangle $B$ in $\lbox$. Such a maximal partition may not be unique; we just fix one. We refer to this partition as the ``knot decomposition" of $P$. Let $\mc{K}$\label{notation:script-K} be the collection of all $K \sse \plaquettes$\label{notation:K} which appear in the knot decomposition of some $P \sse \plaquettes$. We refer to the elements $K \in \mc{K}$ as ``knots".
\end{defn}

\begin{remark}
The definition of knot decomposition given here is slightly different than the one given at the end of \cite[Section 4.1]{Cao2020}. The slightly more complicated definition of \cite{Cao2020} was needed for precisely computing Wilson loop expectations, which was the main focus of \cite{Cao2020}. In the present paper, we are just trying to obtain upper bounds (on correlations), and so we can make do with simpler definition that is given.
\end{remark}

The following lemma bounds the number of knots. This is essentially \cite[Lemma 4.3.4]{Cao2020}. However, because the definition of knot decomposition is slightly different than the one given in \cite{Cao2020}, we give a proof in Appendix \ref{appendix:knot-upper-bound}.

\begin{lemma}[Cf. Lemma 4.3.4 of \cite{Cao2020}]\label{lemma:knot-number-bound}
Let $p \in \plaquettes$. For any $m \geq 1$, the number of knots $K \in \mc{K}$ of size $m$ which contain $p$ is at most $(10^{24})^m$.
\end{lemma}

The following lemma is a slight generalization of \cite[Lemma 4.2.21]{Cao2020}. The point of the definitions of ``well separated" and ``good rectangle" (Definitions \ref{def:well-separated} and \ref{def:good-rectangle}) is to have this lemma.

\begin{lemma}\label{lemma:vortices-well-separated-gauge-fix-decomp}
Let $P_1, P_2 \sse \plaquettes$. Suppose $P_1, P_2$ are well separated by a good rectangle $B$ in $\lbox$. Let $T$ be a spanning tree of $\oneskel$ which contains spanning trees of $S_2(B), S_2^c(B)$, and $\partial S_2(B)$. There is a bijection between the set of $\sigma \in GF(T)$ such that $\supp(\sigma) = P_1 \cup P_2$, and the set of tuples $(\sigma^1, \sigma^2)$ such that $\sigma^i \in GF(T)$, $\supp(\sigma^i) = P_i$, $i = 1, 2$. Moreover, if $\sigma$ is mapped to $(\sigma^1, \sigma^2)$, then $\sigma = \sigma^1 \sigma^2$, $\sigma^1 = \groupid$ on $S_2^c(B)$, $\sigma^2 = \groupid$ on $S_2(B)$. Consequently, for all $p \in S_2(B)$, $\sigma^1_p = \sigma_p$, and for all $p \in S_2^c(B)$, $\sigma^2_p = \sigma_p$.
\end{lemma}
\begin{proof}
This follows by \cite[Lemma 4.1.9]{Cao2020} (see the proof of \cite[Lemma 4.1.21]{Cao2020}).
\end{proof}

Using Lemma \ref{lemma:vortices-well-separated-gauge-fix-decomp}, we can prove the next lemma, which will be important in constructing an event $E(B_1,  B_2)$ and an associated swapping map $T$.

\begin{lemma}\label{lemma:bijection-existence}
Let $P_1, P_2 \sse \plaquettes$ be plaquette sets, and suppose that $P_1$ is well separated from $P_2$ by a good rectangle $B$ in $\lbox$. Then there is a bijection between 
\[ \{\psi \in \homspace : \supp(\psi) \sse P_1 \cup P_2\} \]
and 
\[ \{(\psi_1, \psi_2) \in \Omega^2 : \supp(\psi_i) \sse P_i \text{ for $i = 1, 2$}\}. \]
Moreover, suppose that $\psi$ is mapped to $(\psi_1, \psi_2)$ by the bijection. Then for any $p \in \plaquettes$, we have that $\varphi_\beta(\psi(\xi_p)) = \varphi_\beta(\psi_1(\xi_p)) \varphi_\beta(\psi_2(\xi_p))$. We also have that
\[ \supp(\psi) = \supp(\psi_1) \cup \supp(\psi_2). \]
Finally, let $B_0$ be a rectangle in $\lbox$, and suppose that $B_0 \sse P_i$ for some $i = 1, 2$. Then for any loop $\wloop$ in $B_0$, and any path $\ell$ from $x_0$ to the initial vertex of $\wloop$, we have that $\psi(\ell^{-1} \wloop \ell)$ is conjugate to $\psi_i(\ell^{-1} \wloop \ell)$.
\end{lemma}
\begin{proof}
Let
\[ E := \{\psi \in \homspace : \supp(\psi) \sse P_1 \cup P_2\},\]
and for $i = 1, 2$, let
\[ E_i := \{\psi \in \homspace :  \supp(\psi) \sse P_i\}. \]
We want to construct a bijection between $E$ and $E_1 \times E_2$. Fix $T_B$, a spanning tree of $\oneskel$ which contains spanning trees of $S_2(B), S_2^c(B)$ and $\ptl S_2(B)$ (to construct such a $T_B$, first take a spanning tree $\tilde{T}$ of $\ptl S_2(B)$, and then extend it to spanning trees $T_1, T_2$ of $S_2(B)$, $S_2^c(B)$ respectively, and then define $T_B := T_1 \cup T_2$). By Lemma \ref{lemma:gauge-fixed-bijection-homomorphism}, the set $E$ is in bijection with the set $F := \{\sigma \in GF(T_B) : \supp(\sigma) \sse  P_1 \cup P_2\}$. By Lemma \ref{lemma:vortices-well-separated-gauge-fix-decomp} and the assumption that $P_1$ is well separated from $P_2$ by the good rectangle $B$, the set $F$ is in bijection with the set $F_1 \times F_2$, where
\[ F_i := \{\sigma \in GF(T_B) : \supp(\sigma) \sse P_i\}, ~~ i = 1, 2,   \]
and moreover if $\sigma$ is mapped to $(\sigma^1, \sigma^2)$ under this bijection, then $\sigma = \sigma^1 \sigma^2$, $\sigma^1 = \groupid$ on $S_2^c(B)$, and $\sigma^2 = \groupid$ on $S_2(B)$. Next, by Lemma \ref{lemma:gauge-fixed-bijection-homomorphism}, for $i = 1, 2$, $F_i$ is in bijection with $E_i$.  We can thus obtain a bijection between $E$ and $E_1 \times E_2$, by composing all the previously mentioned bijections: 
\[ E \leftrightarrow F \leftrightarrow F_1 \times F_2 \leftrightarrow E_1 \times E_2. \]
The various claimed properties of this bijection follow because (a) by Lemma \ref{lemma:gauge-fixed-bijection-homomorphism}, the bijections between $E  \leftrightarrow F$ and $E_i \leftrightarrow F_i$ preserve conjugacy, and (b) in the bijection between $F \leftrightarrow F_1 \times F_2$, we have that if $\sigma$ is mapped to $(\sigma^1, \sigma^2)$, then $\sigma_p = \sigma^i_p$ for $p \in P_i$ and $\sigma_p = \groupid = \sigma^1_p \sigma^2_p$ for $p \notin (P_1 \cup P_2)$, and if $B_0 \sse P_i$ for some $i = 1, 2$, then for any loop $\wloop$ in $B_0$, we have that $\sigma_\wloop = \sigma^i_\wloop$.
\end{proof}

Lemma \ref{lemma:bijection-existence} immediately implies the following corollary, whose proof is omitted.

\begin{cor}\label{cor:bijection-existence}
Let $P_1, \ldots, P_k \sse \plaquettes$ be plaquette sets, such that for all $1 \leq i \leq k-1$, $P_i$ is well separated from $P_{i+1} \cup \cdots \cup P_k$ by a good rectangle $B_i$ in $\lbox$. Then there is a bijection between 
\[ \{\psi \in \homspace : \supp(\psi) \sse P_1 \cup \cdots \cup P_k\} \] 
and
\[ \{(\psi_i, i \in [k]) \in \Omega^k :  \supp(\psi_i) \sse P_i \text{ for all $i \in [k]$}\}. \]
Moreover, suppose that $\psi$ is mapped to $(\psi_i, i \in [k])$ by the bijection. Then for any $p \in \plaquettes$, we have that $\varphi_\beta(\psi(\xi_p)) = \prod_{i \in [k]} \varphi_\beta(\psi_i(\xi_p))$. We also have that
\[ \supp(\psi)  = \bigcup_{i \in [k]} \supp(\psi_i). \]
Finally, let $B_0$ be a rectangle in $\lbox$, and suppose that $B_0 \sse P_i$ for some $i \in [k]$. Then for any loop $\wloop$ in $B_0$, and any path $\ell$ from $x_0$ to the initial vertex of $\wloop$, we have that $\psi(\ell^{-1} \wloop \ell)$ is conjugate to $\psi_i(\ell^{-1} \wloop \ell)$.
\end{cor}

\begin{defn}\label{def:fix-bijection}
Given a plaquette set $P \sse \plaquettes$, let $P = K_1 \cup \cdots \cup K_m$ be the knot decomposition of $P$. Let $\Theta(P)$ be a bijection corresponding to $K_1, \ldots, K_m$ as in Corollary \ref{cor:bijection-existence}. Such a bijection may not be unique; we just fix one.
\end{defn}

\section{Construction of the swapping map}\label{section:construction-of-swapping-map}

In this section, we construct an event $E(B_1,  B_2)$ and an associated swapping  map $T$, by using the notions introduced in Section \ref{section:knot-decomposition}. Throughout  this section, fix rectangles $B_1, B_2 \sse \lbox$.

\begin{defn}\label{def:E-gamma-1-gamma-2}
Let $E(B_1,  B_2) \sse \homspace^2$ be the set of pairs $(\psi_1, \psi_2)$ such that the following  holds.
Take the knot decomposition 
\[ \supp(\psi_1) \cup \supp(\psi_2) \cup B_1 \cup B_2 = K_1 \cup \cdots \cup K_m. \]
For $i = 1, 2$, let $j_i \in [m]$ be such that $B_i \sse K_{j_i}$ (the existence of such $j_i$ follows because $B_i$ is a rectangle, and thus cannot be divided into separate knots, by the definition of knot decomposition). Then the condition is that $j_1 \neq j_2$ (i.e. $B_1, B_2$ are contained in different knots).
\end{defn}

We proceed to define a map $T : E(B_1,  B_2) \ra E(B_1,  B_2)$ with the properties required in Definition \ref{def:swapping-map}. 

\begin{defn}\label{def:T}
Let $(\psi_1, \psi_2) \in E(B_1,  B_2)$. As in the definition of $E(B_1,  B_2)$, take the knot decomposition
\[ P := \supp(\psi_1) \cup \supp(\psi_2) \cup B_1 \cup B_2 = K_1 \cup \cdots \cup K_m,  \]
and also let $j_2$ be as defined in that definition. Recalling Definition \ref{def:fix-bijection}, let $\Theta = \Theta(P)$. Define the map $\swap$ from the set 
\[ \{(\psi_i, i \in [m]) \in \homspace^m :  \supp(\psi_i) \sse K_i \text{ for all $i \in [m]$}\}^2 \]
to itself as follows. Map
\[ ((\psi_{1, i}, i \in [m]), (\psi_{2, i}, i \in [m])) \mapsto  ((\tilde{\psi}_{1, i}, i \in [m]), (\tilde{\psi}_{2, i}, i \in [m])), \]
where
\[ \tilde{\psi}_{1, i} := \begin{cases} \psi_{1, i} & i \neq j_2 \\ \psi_{2, j_2} & i = j_2 \end{cases}, \]
\[ \tilde{\psi}_{2, i} := \begin{cases} \psi_{2, i} & i \neq j_2 \\ \psi_{1, j_2} & i = j_2 \end{cases}.\]
In other words, $\swap$ simply swaps the $j_2$-th coordinate. Let $\swap_1, \swap_2$ be the coordinate functions of $\swap$, i.e. they are such that $\swap = (\swap_1, \swap_2)$.  Finally, define $T : E(B_1,  B_2) \ra E(B_1,  B_2)$ by
\[ (\psi_1, \psi_2) \mapsto \big(\Theta^{-1}(\swap_1(\Theta(\psi_1), \Theta(\psi_2))), \Theta^{-1}(\swap_2(\Theta(\psi_1), \Theta(\psi_2)))\big).  \]
This definition may not be very illuminating when written out, but just think of $T$ as swapping the $j_2$-th coordinate (as in the definition of $\swap$), once we have identified $\psi_1, \psi_2$ with the tuples $(\psi_{1, i}, i \in [m])$, $(\psi_{2, i}, i \in [m])$ (and this identification is what $\Theta$ is doing).
\end{defn}


The next sequence of lemmas combine to show that $T$ has all the desired properties.

\begin{lemma} \label{lemma:Tisinvolution}
The map $T$ is an involution, i.e. for any $(\psi_1, \psi_2) \in E(B_1,  B_2)$, we have that $T^2(\psi_1, \psi_2) = (\psi_1, \psi_2)$. As an immediate consequence, $T$ is a bijection.
\end{lemma}
\begin{proof}
If we think of $T$ as simply swapping the $j_2$-th coordinate, then it's clear that $T$ is an involution. This is the idea underlying the proof we proceed to give. Let $T_1, T_2$ be the coordinate functions of $T$. By the definition of $T$ and Corollary \ref{cor:bijection-existence}, we have that
\[ \supp(\psi_1) \cup \supp(\psi_2)  = \supp(T_1(\psi_1, \psi_2))\cup \supp(T_2(\psi_1, \psi_2)). \]
Let $\tilde{\psi}_i = T_i(\psi_1, \psi_2)$ for $i = 1, 2$.  The above implies that we get the exact same knot decomposition
\[ \supp(\tilde{\psi}_1)\cup \supp(\tilde{\psi}_2) \cup B_1 \cup B_2 = K_1 \cup \cdots \cup K_m. \]
Therefore in obtaining $T(\tilde{\psi}_1, \tilde{\psi}_2)$, we use the exact same maps $\Theta, \swap$ as were used to define $T(\psi_1, \psi_2)$. Now by definition, we have that
\[ T_1(\tilde{\psi}_1, \tilde{\psi}_2) = \Theta^{-1}(\swap_1(\Theta(\tilde{\psi}_1), \Theta(\tilde{\psi}_2))), \]
\[ \tilde{\psi}_1 = T_1(\psi_1, \psi_2) = \Theta^{-1}(\swap_1(\Theta(\psi_1), \Theta(\psi_2))), \]
\[ \tilde{\psi}_2 = T_2(\psi_1, \psi_2) = \Theta^{-1}(\swap_2(\Theta(\psi_1), \Theta(\psi_2))).  \]
Let $\eta_i = \swap_i(\Theta(\psi_1), \Theta(\psi_2))$ for $i = 1, 2$. Note that $(\eta_1, \eta_2) = \swap(\Theta(\psi_1), \Theta(\psi_2))$. Combining the above three displays, we obtain
\[ T_1(\tilde{\psi}_1, \tilde{\psi}_2) = \Theta^{-1}(\swap_1(\eta_1, \eta_2)) = \Theta^{-1}(\swap_1(\swap(\Theta(\psi_1), \Theta(\psi_2)))).  \]
Next, observe that $\swap$ is by construction an involution, from which we further obtain
\[  T_1(\tilde{\psi}_1, \tilde{\psi}_2) = \Theta^{-1}(\Theta(\psi_1)) = \psi_1.  \]
The same argument shows that $T_2(\tilde{\psi}_1, \tilde{\psi}_2) = \psi_2$. The desired result now follows.
\end{proof}

\begin{lemma}
For any $(\psi_1, \psi_2) \in E(B_1, B_2)$, we have that
\[\begin{split}
\prod_{p \in \plaquettes} \varphi_\beta&(\psi_1(\xi_p)) \prod_{p \in \plaquettes} \varphi_\beta(\psi_2(\xi_p)) = \\
&\prod_{p \in \plaquettes} \varphi_\beta(T_1(\psi_1, \psi_2)(\xi_p)) \prod_{p \in \plaquettes} \varphi_\beta(T_2(\psi_1, \psi_2)(\xi_p)) . 
\end{split}\]
As an immediate consequence, $\nu_{\lbox, \beta}^{\otimes 2}(T(\psi_1, \psi_2)) = \nu_{\lbox, \beta}^{\otimes 2}((\psi_1, \psi_2))$ (recall Definition \ref{def:nu}).
\end{lemma}
\begin{proof}
Let all notation be as in Definition \ref{def:T}. For $i = 1, 2$, let $\Theta(\psi_i) = (\psi_{i, j}, j \in [m])$. By Corollary \ref{cor:bijection-existence}, we have that
\[ \prod_{p \in \plaquettes} \varphi_\beta(\psi_i(\xi_p)) = \prod_{p \in \plaquettes} \prod_{j \in [m]} \varphi_\beta(\psi_{i, j}(\xi_p)), ~~ i = 1, 2. \]
It follows that
\[\begin{split}
\prod_{p \in \plaquettes} \varphi_\beta(\psi_1(\xi_p))& \prod_{p \in \plaquettes} \varphi_\beta(\psi_2(\xi_p)) = \\
&\prod_{p \in \plaquettes} \bigg( \varphi_\beta(\psi_{2, j_2}(\xi_p)) \prod_{\substack{j \neq j_2}} \varphi_\beta(\psi_{1, j}(\xi_p))\bigg)  ~\times \\
& \prod_{p \in \plaquettes} \bigg( \varphi_\beta(\psi_{1, j_2}(\xi_p)) \prod_{\substack{j \neq j_2}} \varphi_\beta(\psi_{2, j}(\xi_p)) \bigg).
\end{split}\]
By Corollary \ref{cor:bijection-existence} and the definition of $T$, the right hand side above is exactly
\[ \prod_{p \in \plaquettes} \varphi_\beta(T_1(\psi_1, \psi_2)(\xi_p)) \prod_{p \in \plaquettes} \varphi_\beta(T_2(\psi_1, \psi_2)(\xi_p)), \]
and thus the desired result follows.
\end{proof}

\begin{lemma}
Let $k_1, k_2 \geq 1$, and let $f_1 : G^{k_1} \ra \C, f_2 : G^{k_2} \ra \C$ be conjugacy invariant functions. For $i = 1, 2$, let $\wloop_1^{(i)}, \ldots, \wloop_{k_i}^{(i)}$ be closed loops contained in $B_i$. For any $(\psi_1, \psi_2) \in E(B_1, B_2)$, the following holds. Let $\tilde{\psi}_1 = T_1(\psi_1, \psi_2)$. Then we have that
\[\begin{split}
f_1(&\psi_1(\xi_{\wloop_1^{(1)}}), \ldots, \psi_1(\xi_{\wloop_{k_1}^{(1)}})) f_2 (\psi_2(\xi_{\wloop^{(2)}_1}), \ldots, \psi_2(\xi_{\wloop^{(2)}_{k_2}})) = \\
& f_1(\tilde{\psi}_1(\xi_{\wloop_1^{(1)}}), \ldots, \tilde{\psi}_1(\xi_{\wloop_{k_1}^{(1)}})) f_2(\tilde{\psi}_1(\xi_{\wloop^{(2)}_1}), \ldots, \tilde{\psi}_1(\xi_{\wloop^{(2)}_{k_2}})) .
\end{split}\]
\end{lemma}
\begin{proof}
Let all notation be as in Definition \ref{def:T}. For $i = 1, 2$, let $\Theta(\psi_i) = (\psi_{i, j}, j \in [m])$. By Corollary \ref{cor:bijection-existence}, we have that
\beq\label{eq:wilson-loop-localize-1} \psi_i(\xi_{\gamma_l^{(i)}}) \text{ is conjugate to } \psi_{i, j_i}(\xi_{\gamma_l^{(i)}}), ~~ l \in [k_i], ~~ i = 1, 2. \eeq
Now by construction, we have that
\[ \Theta(\tilde{\psi}_1) = (\psi_{1, 1}, \ldots, \psi_{1, j_2 - 1}, \psi_{2, j_2}, \psi_{1, j_2 + 1}, \ldots, \psi_{1, m}).\]
Thus by Corollary \ref{cor:bijection-existence} and equation \eqref{eq:wilson-loop-localize-1}, we have that
\[ \tilde{\psi}_1(\xi_{\wloop_l^{(1)}}) \text{ is conjugate to } \psi_{1, j_1}(\xi_{\wloop_l^{(1)}}) \text{ is conjugate to } \psi_1(\xi_{\wloop_l^{(1)}}), ~~ l \in [k_1]\]
\[ \tilde{\psi}_1(\xi_{\wloop_l^{(2)}}) \text{ is conjugate to } \psi_{2, j_2}(\xi_{\wloop_l^{(2)}}) \text{ is conjugate to } \psi_2(\xi_{\gamma_l^{(2)}}), ~~ l \in [k_2]. \]
Since $f_1, f_2$ are conjugacy invariant, the desired result now follows.
\end{proof}

The previous few lemmas combine to immediately imply the following corollary.

\begin{cor}\label{cor:swapping-map}
For any functions $h_1, h_2 : \homspace \ra \C$ as in Proposition \ref{prop:existence-of-swapping-map}, $T$ is a swapping map with respect to $(E(B_1, B_2), h_1, h_2, \nu_{\lbox, \beta})$.
\end{cor}

\section{Probability bound}\label{section:probability-bound}

In this section, we prove the probability bound from Proposition \ref{prop:existence-of-swapping-map}.

\begin{prop}\label{prop:percolation-probability-bound}
Let $\beta, L, B_1, B_2$ be as in Theorem \ref{thm:main-general}. Let $E(B_1, B_2)$ be defined using $B_1, B_2$, as in Definition \ref{def:E-gamma-1-gamma-2}. Let $\Psi_1, \Psi_2 \stackrel{i.i.d.}{\sim} \nu_{\lbox, \beta}$. Then
\[ \p((\Psi_1, \Psi_2) \notin E(B_1, B_2)) \leq  2 (4 \cdot 10^{24} |G|^2)^{|B_1| + |B_2|} e^{-(\beta / 2) \Delta_G (L-1)}.\]
\end{prop}

As we will see, the proof is essentially a Peierls argument.~To begin towards the proof, we make the following definitions.

\begin{defn}
Let $P \sse \plaquettes$. Given $\psi_1, \psi_2 \in \homspace$, define
\[ \supp_P(\psi_1, \psi_2) := \supp(\psi_1) \cup \supp(\psi_2) \cup P. \]
\end{defn}

\begin{defn}
Let $P_0, P \sse \plaquettes$.  Define
\[ \Phi^{(2)}_{P_0}(P) := \sum_{\substack{ \psi_1, \psi_2 \in \homspace \\ \supp_{P_0}(\psi_1, \psi_2) = P}} \prod_{p \in \plaquettes} \varphi_\beta(\psi_1(\xi_p)) \varphi_\beta(\psi_2(\xi_p)).  \]
Define $\Phi^{(2)}(P) := \Phi^{(2)}_{\varnothing}(P)$, i.e. if there is no subscript, then by default we take $P_0$ to be the empty set.
\end{defn}

The following lemma shows that the function $\Phi^{(2)}_P$ factors according to the notion of well separated (recall Definition \ref{def:well-separated}).

\begin{lemma}\label{lemma:Phi-2-factorization}
Let $P_1, P_2 \sse \plaquettes$ be well separated by a good rectangle $B$ in $\lbox$. Then
\[ \Phi^{(2)}(P_1 \cup P_2) = \Phi^{(2)}(P_1) \Phi^{(2)}(P_2). \]
Additionally, let $P_0 \sse \plaquettes$, and suppose that $P_0 \sse P_1$ (resp. $P_0 \sse P_2$). Then
\[ \Phi^{(2)}_{P_0}(P_1 \cup P_2) = \Phi^{(2)}_{P_0}(P_1) \Phi^{(2)}(P_2) \text{ (resp. $\Phi^{(2)}(P_1) \Phi^{(2)}_{P_0}(P_2)$)}. \]
\end{lemma}
\begin{proof}
To show the first identity, we need to show that there is a bijection between the sets
\[ E := \{(\psi_1, \psi_2) \in \homspace^2 : \supp(\psi_1, \psi_2) = P_1 \cup P_2 \}\]
and $E_1 \times E_2$, where
\[ E_i := \{(\psi_{1, i}, \psi_{2, i}) \in \homspace^2 : \supp(\psi_{1, i}, \psi_{2, i}) = P_i \}, ~~ i = 1,2, \]
and moreover, if $(\psi_1, \psi_2)$ is mapped to $((\psi_{1, 1}, \psi_{2, 1}), (\psi_{1, 2}, \psi_{2, 2}))$ by the bijection, then we have that
\[\begin{split}
\prod_{p \in \plaquettes}& \varphi_\beta(\psi_1(\xi_p)) \varphi_\beta(\psi_2(\xi_p)) = \\
&\prod_{p \in \plaquettes} \varphi_\beta(\psi_{1, 1}(\xi_p)) \varphi_\beta(\psi_{2, 1}(\xi_p))
\prod_{p \in \plaquettes} \varphi_\beta(\psi_{1, 2}(\xi_p)) \varphi_\beta(\psi_{2, 2}(\xi_p)). 
\end{split}\]
Let $\Theta$ be a bijection as in Corollary \ref{cor:bijection-existence} corresponding to $P_1, P_2$. Given $(\psi_1, \psi_2) \in E$ and $i = 1, 2$, let $\Theta(\psi_i) = (\psi_{i, 1}, \psi_{i, 2})$. Define the map $E \ra E_1 \times E_2$ by $(\psi_1, \psi_2) \mapsto ((\psi_{1, 1}, \psi_{2, 1}), (\psi_{1, 2}, \psi_{2, 2}))$. The fact that this is a bijection with the required properties follows by Corollary \ref{cor:bijection-existence}. The second identity may be similarly argued.
\end{proof}

By repeated applications of Lemma  \ref{lemma:Phi-2-factorization}, we can obtain the following corollary. The proof is omitted.

\begin{cor}\label{cor:Phi-2-factorization}
Let $P_0 \sse \plaquettes$. Let $K \sse \plaquettes$ be a knot such that $K \supseteq P_0$. If $P \sse \plaquettes$ is such that $K$ appears in the knot decomposition of $P$, then
\[ \Phi^{(2)}_{P_0}(P) = \Phi^{(2)}_{P_0}(K)  \Phi^{(2)}(P - K). \]
\end{cor}

In the usual Peierls argument for long range order of the Ising model at low temperatures, one shows that the presence of a large contour is exponentially unlikely in the length of the contour. The following two lemmas combine to give the analogous statement for our setting.

\begin{lemma}\label{lemma:knot-appear-prob-bound-by-Phi}
Let $P_0 \sse \plaquettes$. Let $K \sse \plaquettes$ be a knot such that $K \supseteq P_0$. Let $\Psi_1, \Psi_2 \stackrel{i.i.d.}{\sim} \nu_{\lbox, \beta}$, and let $F_K$ be the event that $K$ appears in the knot decomposition of $\supp_{P_0}(\Psi_1, \Psi_2)$. Then $\p(F_K) \leq \Phi^{(2)}_{P_0}(K)$.
\end{lemma}
\begin{proof}
For notational brevity, given $P \sse \plaquettes$, let $K \in P$ to mean that $K$ is in the knot decomposition of $P$. We have that
\[ \p(F_K) = \frac{\sum_{\psi_1, \psi_2 \in \homspace} \ind(K \in \supp_{P_0}(\psi_1, \psi_2)) \prod_{p \in  \plaquettes} \varphi_\beta(\psi_1(\xi_p)) \varphi_\beta(\psi_2(\xi_p))}{\sum_{\psi_1, \psi_2 \in \homspace} \prod_{p \in  \plaquettes} \varphi_\beta(\psi_1(\xi_p)) \varphi_\beta(\psi_2(\xi_p))}. \]
Observe that the denominator is equal to $\sum_{P \sse \plaquettes} \Phi^{(2)}(P)$, while the numerator is equal to
\[ \sum_{\substack{P \sse \plaquettes \\ K \in P}} \Phi^{(2)}_{P_0}(P) = \Phi^{(2)}_{P_0}(K) \sum_{\substack{P \sse \plaquettes \\ K \in P}} \Phi^{(2)}(P - K), \]
where we have applied Corollary \ref{cor:Phi-2-factorization}. We may further bound the right hand side above by
\[ \Phi^{(2)}_{P_0}(K) \sum_{P \sse \plaquettes} \Phi^{(2)}(P) . \]
The desired result now follows by combining the previous observations.
\end{proof}

\begin{lemma}\label{lemma:Phi-bound}
Let $P_0, P \sse \plaquettes$. Then
\[ \Phi^{(2)}_{P_0}(P) \leq 4^{|P|} |G|^{2|P|} e^{-\beta \Delta_G |P  - P_0|}.\]
\end{lemma}
\begin{proof}
In order for $\psi_1, \psi_2 \in \homspace$ to be such that $\supp_{P_0}(\psi_1, \psi_2) = P$, we must have that $\supp(\psi_1), \supp(\psi_2) \sse P$. Thus there are at most $4^{|P|} = 2^{|P|} 2^{|P|}$ possible choices of $(\supp(\psi_1), \supp(\psi_2))$. Now fix $S_1, S_2 \sse P$ (such that $S_1 \cup S_2 \cup {P_0} = P$). It remains to show that
\[\begin{split}
\sum_{\substack{\psi_1, \psi_2 \in \homspace \\ \supp(\psi_i) = S_i, i = 1,2}} \prod_{p \in P} \varphi_\beta(\psi_1(\xi_p))& \varphi_\beta(\psi_2(\xi_p)) \leq \\
&|G|^{2|P|} e^{-\beta \Delta_G |P  - P_0|}. 
\end{split}\]
First, note that for any $\psi_1, \psi_2 \in \homspace$ such that $\supp(\psi_i) = S_i$ for $i = 1, 2$, we have that for any $p \in (S_1 \cup S_2) - P_0 = P - P_0$,
\[ \varphi_\beta(\psi_1(\xi_p))  \varphi_\beta(\psi_2(\xi_p)) \leq e^{-\beta \Delta_G}. \]
It thus follows that
\[ \prod_{p \in P} \varphi_\beta(\psi_1(\xi_p)) \varphi_\beta(\psi_2(\xi_p)) \leq e^{-\beta \Delta_G |P  - P_0|}. \]
To finish, note that by Lemma \ref{lemma:homomorphism-number-bound}, for $i = 1, 2$, the number of homomorphisms $\psi \in \homspace$ such that $\supp(\psi) = S_i$ is at most $|G|^{|S_i|} \leq |G|^{|P|}$.
\end{proof}

The following lemma shows that if a knot $K$ contains plaquettes which are very far apart from each other, then the knot itself must be large (intuitively, this is true because knots are in some sense ``connected"). This is the last ingredient needed for our Peierls argument.

\begin{lemma}\label{lemma:knot-size-lower-bound}
Let $L \geq 0$. Let $P_1, P_2 \sse \plaquettes$ be plaquette sets such that the $\ell^\infty$ distance between any vertex of $P_1$ and any vertex of $P_2$ is at least $L$.  
If $K \sse \plaquettes$ is a knot such that $K \supseteq P_1 \cup P_2$, then $|K| \geq |P_1| + |P_2| + L - 1$.
\end{lemma}
\begin{proof}
By assumption, there is some coordinate direction $i \in [4]$ such that the distance between $P_1$ and $P_2$ in the $i$th coordinate is at least $L$. Thus (without loss of generality) we can assume that there are integers $m_1, m_2 \in \Z$ such that $m_2 - m_1 \geq L$, and for all vertices $x$ of a plaquette in $P_1$, and all vertices $y$ of a plaquette in $P_2$, we have that $x_i \leq m_1$, $y_i \geq m_2$. For $m_1 + 1 \leq k \leq m_2  - 1$, let $R_k$ be the rectangle in  $\lbox$ defined by taking all vertices $x \in \vertices$ such that $x_i \leq k$. Observe that for all $k$, $\ptl S_2(R_k)$ is the set of plaquettes $p \in \plaquettes$ such that all vertices $x$ of $p$ have $x_i = k$. Thus the plaquette sets $(\ptl S_2(R_k), m_1 + 1 \leq k \leq m_2 - 1)$ are mutually disjoint, and moreover $\ptl S_2(R_k)$ is disjoint from $P_1 \cup P_2$ for all $m_1 + 1 \leq k \leq m_2 - 1$. By Lemma \ref{lemma:good-rectangle-examples}, $R_k$ is a good rectangle, and thus by the definition of knot decomposition, we must have that for all $m_1 + 1 \leq k \leq m_2 - 1$, $K \cap \ptl S_2(R_k) \neq \varnothing$. Combining these observations with the assumption that $K \supseteq P_1 \cup P_2$, we obtain that $|K| \geq |P_1| + |P_2| + (m_2 - m_1 - 1 ) \geq |P_1| + |P_2| + L - 1$, as desired.
\end{proof}

We finally have enough to prove Proposition \ref{prop:percolation-probability-bound}. 

\begin{proof}[Proof of Proposition \ref{prop:percolation-probability-bound}]
Let $m_0 := |B_1| + |B_2|$. In order for $(\Psi_1, \Psi_2) \notin E(B_1, B_2)$, there must be a knot $K$ containing the plaquettes of $B_1, B_2$ such that the event $F_K$ occurs. Thus by a union bound and Lemma \ref{lemma:knot-size-lower-bound}, we have that
\[ \p((\Psi_1, \Psi_2) \notin E(B_1, B_2)) \leq \sum_{m=m_0 + L-1}^\infty \sum_{\substack{{\text{$K$ a knot of size $m$}} \\ K \supseteq B_1 \cup B_2}} \p(F_K) .\]
Now combining Lemmas \ref{lemma:knot-number-bound}, \ref{lemma:knot-appear-prob-bound-by-Phi}, and \ref{lemma:Phi-bound}, we further obtain (using the assumption that $\beta$ is large enough to ensure that the geometric series is summable)
\begin{align*}
\p((\Psi_1, \Psi_2) \notin E(B_1, B_2)) &\leq \sum_{m=m_0 + L-1}^\infty (4 \cdot 10^{24} |G|^2)^m e^{-\beta \Delta_G (m - m_0)} \\
&= (4 \cdot 10^{24} |G|^2)^{m_0} \frac{(4 \cdot 10^{24} |G|^2 e^{-\beta \Delta_G})^{L-1}}{1 - 4 \cdot 10^{24} |G|^2 e^{-\beta \Delta_G}}.
\end{align*}
To finish, note that the assumption on $\beta$ implies that
\[ 4 \cdot 10^{24} |G|^2 e^{-(\beta / 2)\Delta_G} \leq 1,\]
and thus we obtain 
\[ \p((\Psi_1, \Psi_2) \notin E(B_1, B_2)) \leq 2 (4 \cdot 10^{24} |G|^2)^{m_0} e^{-(\beta / 2) \Delta_G (L-1)}, \]
as desired.
\end{proof}

\begin{proof}[Proof of Proposition \ref{prop:existence-of-swapping-map}]
This is a direct consequence of Corollary \ref{cor:swapping-map} and Proposition \ref{prop:percolation-probability-bound}.
\end{proof}

\appendix 

\numberwithin{theorem}{section}

\section{Topological facts}

\subsection{Edge configurations and homomorphisms}\label{appendix:edge-configurations-and-homomorphisms}

First, we recall an explicit set of generators of $\pi_1(\oneskel, x_0)$ (for a reference, see e.g. \cite[Section 2.1.7]{STILL1993}). Fix a spanning tree $T$\label{notation:T} of $\oneskel$. For any vertex $x \in \vertices$, let $\upath_x$ denote the unique path in $T$ from $x_0$ to $x$. For any edge $e = (x, y) \in \oneskel$, let $a_e$ be the closed loop obtained by starting at $x_0$, following $\upath_x$ to $x$, then traversing $e = (x, y)$, then following the path $\upath_y$ in reverse, from $y$ to $x_0$. Symbolically, we write
\[\label{notation:a-e} a_e = \upath_x e \upath_y^{-1}. \]
(Note if $e$ is in the spanning tree $T$, and $x$ is closer than $y$ to $x_0$ (in the distance induced by $T$), then $a_e = w_y w_y^{-1}$ is the path which starts at $x_0$, follows the path $\upath_y$ to $y$, and then retraces its steps, following the path $\upath_y$ in reverse, from $y$ to $x_0$. Thus in this case $a_e$ is equivalent to the trivial path.) Then for any closed loop $\wloop = e_1 \cdots e_n$ in $\oneskel$ starting and ending at $x_0$, we have that $\wloop$ is equivalent to $a_{e_1} \cdots a_{e_n}$. Thus if $[a_e] \in \pi_1(\oneskel, x_0)$ denotes the equivalence class containing $a_e$, we have that $\{[a_e], e \in \oneskel - T\}$ is a generating set for $\pi_1(\oneskel, x_0)$.

\begin{lemma}\label{lemma:induced-homomorphism-onto}
For any $\psi \in \homspace$, there exists $\sigma \in G^\edges$ such that $\psi_{x_0}(\sigma) = \psi$.
\end{lemma}
\begin{proof}
Fix a spanning tree $T$ of $\oneskel$. Define $\sigma \in G^\edges$ as follows. For $e \in T$, define $\sigma_e := \groupid$. For $e \in \oneskel - T$, define $\sigma_e := \psi([a_e])$. By construction, we have that for all $e \in \oneskel - T$, $(\psi_{x_0}(\sigma))([a_e]) = \sigma_e = \psi([a_e])$. Since $\{[a_e], e \in \oneskel - T\}$ is a generating set for $\pi_1(\oneskel, x_0)$, it follows that $\psi_{x_0}(\sigma) = \psi$, as desired.
\end{proof}

\begin{lemma}\label{lemma:induced-homomorphism-equal-implies-gauge-equiv}
Let $\sigma, \tau \in G^\edges$. Then $\homsym_{x_0}(\sigma) = \homsym_{x_0}(\tau)$ if and only if there exists a function $h \in G^{\vertices}$ with $h_{x_0} = \groupid$, such that for all edges $e = (x, y) \in \oneskel$, we have 
\[ \sigma_e = h_x \tau_e h_y^{-1}. \]
\end{lemma}
\begin{proof}
We prove the nontrivial direction. Suppose that $\psi_{x_0}(\sigma) = \psi_{x_0}(\tau)$. We define $h \in G^{\vertices}$ as follows. Fix a spanning tree $T$ of $\oneskel$. First, as required, $h_{x_0} := \groupid$. Now for any edge $e = (x_0, x) \in T$, define $h_x$ so that
\[ \sigma_e = h_{x_0} \tau_e h_x^{-1},  \]
i.e.
\[ h_x := \sigma_e^{-1} \tau_e. \]
More generally, for any $x \in \vertices$, suppose $\upath_x = e_1 \cdots e_n$ (recall the notation that $\upath_x$ is the unique path in $T$ from $x_0$ to $x$). Define
\[ h_x := (\sigma_{e_1} \cdots \sigma_{e_n})^{-1} \tau_{e_1} \cdots \tau_{e_n}.\]
We now show that $h$ is as required. Fix an edge $e = (x, y)$. Suppose first that $e \in T$. Without loss of generality, suppose that $y$ is further from the root $x_0$ of $T$ than $x$, so that $\upath_y = \upath_x e$. If $\upath_x = e_1 \cdots e_n$, then
\[ h_x \tau_e h_y^{-1} = (\sigma_{e_1} \cdots \sigma_{e_n})^{-1} \tau_{e_1} \cdots \tau_{e_n} \tau_e (\tau_{e_1} \cdots \tau_{e_n} \tau_e)^{-1} \sigma_{e_1} \cdots \sigma_{e_n} \sigma_e = \sigma_e, \]
as desired. Now suppose $e \in \oneskel \backslash T$. Let $\upath_x = e_1 \cdots e_n$, $\upath_y = f_1 \cdots f_m$. Then
\begin{align*} 
h_x \tau_e h_y^{-1} &= (\sigma_{e_1} \cdots \sigma_{e_n})^{-1} \tau_{e_1} \cdots \tau_{e_n} \tau_e (\tau_{f_1} \cdots \tau_{f_m})^{-1} \sigma_{f_1} \cdots \sigma_{f_m} \\
&= (\sigma_{e_1} \cdots \sigma_{e_n})^{-1} (\homsym_{x_0}(\tau))([a_e]) \sigma_{f_1} \cdots \sigma_{f_m}.
\end{align*}
To finish, we want to show
\[(\sigma_{e_1} \cdots \sigma_{e_n})^{-1} (\homsym_{x_0}(\tau))([a_e]) \sigma_{f_1} \cdots \sigma_{f_m} = \sigma_e. \]
If we move all the $\sigma$'s to the right hand side, we see that we need to show
\[ (\homsym_{x_0}(\tau))([a_e]) = (\homsym_{x_0}(\sigma))([a_e]), \]
which is true by assumption.
\end{proof}

\begin{lemma}[Lemma 4.1.3 of \cite{Cao2020}]\label{lemma:different-gauge-transform-implies-different-edge-configuration}
Let $\sigma \in G^\edges$, and $h \in G^{\vertices}$ with $h_{x_0} = \groupid$. Let $\tau \in G^\edges$ be the edge configuration given by $\tau_e := h_x \sigma_e h_y^{-1}$ for each $e = (x, y) \in \edges$. If there exists $x \in \vertices$ such that $h_x \neq \groupid$, then $\sigma \neq \tau$.
\end{lemma}

\begin{proof}[Proof of Lemma \ref{lemma:psi-many-to-one}]
This follows by Lemmas \ref{lemma:induced-homomorphism-onto}, \ref{lemma:induced-homomorphism-equal-implies-gauge-equiv}, and \ref{lemma:different-gauge-transform-implies-different-edge-configuration}.
\end{proof}

\begin{proof}[Proof of Lemma \ref{lemma:gauge-fixed-bijection-homomorphism}]
Let $\psi \in \homspace$. The existence of $\sigma \in GF(T)$ such that $\psi_{x_0}(\sigma) = \psi$ follows from the proof of Lemma \ref{lemma:induced-homomorphism-onto}. For uniqueness, observe that for $e \in \oneskel - T$, we must have $\sigma_e = (\psi_{x_0} \sigma)([a_e]) = \psi([a_e])$. Thus $\psi_{x_0}$ is a bijection. The various properties of $\psi_{x_0}$ follow by Lemma \ref{lemma:psi-conjugacy-preserve}.
\end{proof}

\subsection{Good rectangles}\label{appendix:good-rectangles}

\begin{proof}[Proof of Lemma \ref{lemma:good-rectangle-examples}]
Case (a) is \cite[Lemma 4.1.20]{Cao2020}. For case (b), first note that by \cite[Lemma 4.1.20]{Cao2020}, $S_2(B)$ is simply connected, since $B$ is a rectangle. Now, suppose that the vertices of $B$ are given by the set $\{x \in \vertices : x_i \leq k\}$ for some $i \in [4]$, $k \in \Z$ (the case $x_i \geq k$ follows by the same argument). This implies that $\ptl S_2(B)$ is the 2-complex obtained from the set of plaquettes $p$ whose vertices all lie in the set $\{x \in \vertices : x_i = k\}$. This implies that $S_2^c(B)$ is also a rectangle (whose vertices are $\{x \in \vertices : x_i \geq k\}$), and thus by \cite[Lemma 4.1.20]{Cao2020} $S_2^c(B)$ is simply connected. To see why $\ptl S_2(B)$ is simply connected, note that it essentially is a rectangle in one lower dimension, i.e. it is a 3D rectangle. Following the proof of \cite[Lemma 4.1.20]{Cao2020} in \cite[Appendix B]{Cao2020}, we may attach to $\ptl S_2(B)$ all 3-cells whose boundary plaquettes are all contained in $\ptl S_2(B)$. This operation does not change the fundamental group (see e.g. \cite[Section 4.1.5]{STILL1993}), and the resulting space is a 3D rectangle in $\R^4$, which is simply connected.
\end{proof}

\subsection{Knot upper bound}\label{appendix:knot-upper-bound}

Most of the following discussion is borrowed from \cite[Section 4.4]{Cao2020}.~Indeed, to prove Lemma \ref{lemma:knot-number-bound}, we will show how to deduce it from \cite[Corollary 4.4.8]{Cao2020}.

\begin{defn}[Cf. Definition 3.2.2 of \cite{Cao2020}]
Given a plaquette set $P \sse \plaquettes$, we may obtain an undirected graph $G(P)$ as follows. The vertices of the graph are the plaquettes of $P$. Place an edge between any two plaquettes $p_1, p_2 \in P$ such that there is a 3-cell $c$ in $\lbox$ which contains both $p_1, p_2$. 

A vortex is a set $\vortex \sse \plaquettes$\label{notation:vortex} such that $G(\vortex)$ is connected. For general plaquette sets $P \sse \plaquettes$, we may partition $G(P)$ into connected components $G_1, \ldots, G_k$, which corresponds to a partition of $P$ into vortices $\vortex_1, \ldots, \vortex_k$, such that $G_i = G(\vortex_i)$ for all $1 \leq i \leq k$. Observe that as the partition of an undirected graph into connected components is unique, the resulting partition of $P$ into vortices is also unique. 

Given a plaquette set $P \sse \plaquettes$, let the unique partition of $P$ into compatible vortices $\vortex_1, \ldots, \vortex_k$ as previously described be called the vortex decomposition of $P$. For each $1 \leq i \leq k$, we say that $\vortex_i$ is a vortex of $P$.
\end{defn}

\begin{defn}
Given a plaquette set $P \sse \plaquettes$, define $B(P)$ as a cube of minimal side length in $\lbox$ such that all plaquettes of $P$ are in $S_2(B(P))$, but not in $\partial S_2(B(P))$. If the choice of $B(P)$ is not unique, fix one such cube. For $P, P' \sse \plaquettes$, define the function
\[ J(P, P') := \begin{cases} 1 & P \cap B(P') \neq \varnothing \text{ or } P' \cap B(P) \neq \varnothing \\ 0 & \text{otherwise} \end{cases} .\]
To be clear, we are slightly abusing notation here by writing $P \cap B(P') \neq \varnothing$; what this means is that there is a plaquette $p \in P$ which is contained in $B(P')$. 
\end{defn}

\begin{defn}
We define a hierarchy of undirected graphs $G^\scale(P)$, for integers $\scale \geq 0$. First, to define $G^0(P)$, consider the vortex decomposition $P = \vortex_1 \cup \cdots \cup \vortex_{n_0}$. Define $P^0_i := V_i$, $1 \leq i \leq n_0$. The vertex set of $G^0(P)$ is $\{P^0_1, \ldots, P^0_{n_0}\}$. The edge set is 
\[ \{ \{P^0_i, P^0_j\} : i \neq j, J(P^0_i, P^0_j) = 1\}. \]
Now suppose for some $\scale \geq 0$, $G^\scale(P)$\label{notation:G-scale-P} is defined, with vertices $P^\scale_1, \ldots, P^\scale_{n_\scale} \sse \plaquettes$ which are compatible (and thus disjoint), and such that $P = P^\scale_1 \cup \cdots \cup P^\scale_{n_\scale}$. To define $G^{\scale+1}(P)$, first let $n_{\scale+1}$ be the number of connected components of $G^\scale(P)$, with connected components given by the partition $I_1 \cup \cdots \cup I_{n_{\scale+1}} = [n_\scale]$. For $1 \leq i \leq n_{\scale+1}$, define
\[ P^{\scale+1}_i := \bigcup_{j \in I_i} P^\scale_j. \]
The vertex set of $G^{\scale+1}(P)$ is $\{P^{\scale+1}_1, \ldots, P^{\scale+1}_{n_{\scale+1}}\}$, and the edge set is
\[ \{ \{P^{\scale+1}_i, P^{\scale+1}_j\} : i \neq j, J(P^{\scale+1}_i, P^{\scale+1}_j) = 1\}. \]
Observe that if $\scale \geq 1$ is such that $n_\scale = 1$, then $G^{\scale-1}(P)$ is connected. 
\end{defn}

\begin{defn}
For $P \sse \plaquettes$, define $\scale^*(P) := \min \{ \scale : n_\scale = 1\}$. If $n_\scale > 1$ for all $\scale$, define $\scale^*(P) = \infty$. Let $\mc{D}$ be the collection of $P \sse \plaquettes$ such that $\scale^*(P) < \infty$, and such that for all $\scale \leq \scale^*(P)$, any vertex of $G^\scale(K)$ is of size at least $2^\scale$. Observe that if $P \in \mc{D}$ and $|P| = m$, then $\scale^*(P) \leq \floor{\log_2 m}$, and consequently $G^{\floor{\log_2 m} - 1}(P)$ is connected. Now define
\[\label{notation:A-m-scale} \mc{A}(m, \scale) := \{ P \in \mc{D} : \abs{P} = m, G^\scale(P) \text{ is connected}\}. \]
By the previous observation, note if $\scale \geq \floor{\log_2 m} - 1$, then $\mc{A}(m, \scale) = \mc{A}(m, \floor{\log_2 m} - 1)$.
Now for $p \in \plaquettes$, define
\[\label{notation:A-m-scale-p} \mc{A}(m, \scale, p) := \{P \in \mc{A}(m, \scale) : P \ni p\}. \]
\end{defn}

\begin{lemma}[Corollary 4.4.8 of \cite{Cao2020}]\label{lemma:A-m-s-p-bound}
For all  $m \geq 1$, $0 \leq \scale \leq \floor{\log_2 m} - 1$, $p \in \plaquettes$, we have that
\[ |\mc{A}(m, \scale, p)| \leq (10^{24})^m. \]
\end{lemma}

\begin{lemma}\label{lemma:knot-graph-properties}
Let $K \in \mc{K}$, $\scale \geq 0$. If $n_\scale > 1$, then there are no isolated vertices of $G^\scale(K)$. Consequently, every connected component of $G^\scale(K)$ is of size at least $2$, and thus if also $\scale \leq \scale^*(K)$, then for all vertices $P^\scale_i$ of $G^\scale(K)$, we have $|P^\scale_i| \geq 2^\scale$. Consequently, $K \in \mc{D}$.
\end{lemma}
\begin{proof}
Let $P^\scale_i$ be a vertex of $G^\scale(K)$. Since $n_\scale > 1$ by assumption, we have $K - P^\scale_i \neq \varnothing$. Then by the definition of knot decomposition (Definition \ref{def:knot-decomposition}), there cannot exist a cube which well separates $P^\scale_i$ from $K - P^\scale_i$ (note that such a cube would automatically have side length strictly less than the side length of $\lbox$, and thus by Lemma \ref{lemma:good-rectangle-examples}, it would be a good rectangle). Therefore we must have $B(P^\scale_i) \cap (K - P^\scale_i) \neq \varnothing$, and thus $P^\scale_i$ cannot be an isolated vertex in $G^\scale(K)$. 
\end{proof}

\begin{proof}[Proof of Lemma \ref{lemma:knot-number-bound}]
By Lemma \ref{lemma:knot-graph-properties}, for any knot $K \in \mc{K}$, we have that $K \in \mc{D}$. Thus if $|K| = m$, then $K \in \mc{A}(m, \floor{\log_2 m} - 1)$. The desired result now follows by Lemma \ref{lemma:A-m-s-p-bound}.
\end{proof}


\section{Higgs Boson Large \texorpdfstring{$\kappa$}{k}}\label{appendix:higgs-large-kappa}

In this section, we will adapt our swapping argument technique to establish correlation decay for Wilson loop observables for lattice Yang-Mills theories coupled to a Higgs boson. This is the model treated in the work \cite{Adh2021}. 

In said work, the analysis of the Higgs' boson differed in two cases; namely, the cluster expansion techniques used when $\kappa$ was large or when $\kappa$ was small were different. In this section, we discuss briefly the case that $\kappa$ is large. 

The main new term of the Hamiltonian, when coupled to the Higgs' boson, is a term of the form $ \kappa \sum_{e=(x,y) \in \Lambda_1} \text{Re} [\phi_x \text{Tr}[\rho(\sigma_e)] \phi_y^{-1} - \text{Tr}[\rho(1)] ]$. When $\kappa$ is large, this term tends to force $\sigma_e = 1$ as well as $\phi_x =\phi_y$. 

The new difficulty introduced when applying our swapping argument is that we now need to describe how one will swap the Higgs' boson configuration appropriately. On a high level, the appropriate 'clusters' for the Higgs boson are found by locating the boundaries between unequal $\phi_x$ and $\phi_y$. The swapping map we design will preserve said boundaries as appropriate. In what follows, we go into more detail on the construction. 

\subsection{Introduction}

Our configurations consist of two components. The first is the gauge field $\sigma: \Lambda_1 \to G$ with representation $\rho$. (We abuse notation slightly and use $\Lambda_1$ to represent all edges rather than just the positively oriented ones; similarly, $\Lambda_2$ will denote all of the plaquettes, rather than just the positively oriented ones.) We have considered excitations of this field in the previous sections. The second is a Higgs boson field that acts as a map from the vertices $\Lambda_0 \to H$. $H$ is a subgroup of the multiplicative group of the unit circle. The Hamiltonian that we consider is,
\begin{equation}
    H_{N,\beta,\kappa}(\sigma,\phi)= \sum_{p \in \Lambda_2} \beta \text{Tr}[\rho((\td \sigma)_p) - \rho(1)] + \kappa \sum_{e=(x,y) \in \Lambda_1} \text{Re} [\phi_x \text{Tr}[\rho(\sigma_e)] \phi_y^{-1} - \text{Tr}[\rho(1)] ].
\end{equation}

If we consider computing the expectation of Gauge invariant functions, then we can make the following simplification to the gauge group. We let $H_t$ be the set of elements $h_t$ such that there exists some $g \in G$ such that $\rho(g) = h_t I$. Then, by an appropriate simultaneous gauge transformation of $\sigma$ and $\phi$, we may consider the case that the Higgs field is a map from $\Lambda_0 \to H/H_t$. The benefit of this gauge transformation is that if $\phi_x \text{Tr}[\rho(\sigma_e)] \phi_y^{-1} = \text{Tr}[\rho(1)]$, then necessarily $\phi_x = \phi_y$ and $\sigma_e =1$.

As such, we can have the following definition of support of our configuration.
\begin{defn}
Consider a configuration $\calC = (\sigma,\phi)$.
We let $EE$, the set of our excited edges, be defined as,
\begin{equation}
    EE:=\{ e=(x,y) \in \Lambda_1: \phi_x \ne \phi_y \text{ or } \sigma_e\ne 1 \}.
\end{equation}

The definition of the support of our configuration is as follows,
\begin{equation}
    \supp(\calC) = \{ p \in \Lambda_2: \exists e \in EE \text{ s.t. } e \in \delta p \}.
\end{equation}
Namely, a plaquette $p$ is in our support if there exists an edge $e$ in the set of our excited edges $EE$ such that $e$ is in the boundary of $p$.

Given two configurations $\calC_1$ and $\calC_2$, the joint support of $\calC_1$ and $\calC_2$ will be the union of the supports of $\calC_1$ and $\calC_2$.

\begin{equation}
    \supp_P(\calC_1,\calC_2)= \supp(\calC_1) \cup \supp(\calC_2) \cup P.
\end{equation}
\end{defn}

Our goal in this section is to prove the following theorem,
\begin{theorem}
Consider two boxes $B_1$, $B_2$ separated by $\ell^{\infty}$ distance $L$. Additionally, let $f$ and $g$ be functions such that $f$ only depends on the values of $\sigma$ and $\phi$ in the box $B_1$, while $g$ depends only on the values of $\sigma$ and $\phi$ in the box $B_2$.
Furthermore, assume that transformations of the form $\sigma_e \to \eta_x \sigma_e \eta_y^{-1}$ for some auxiliary field $\eta:\Lambda_0 \to G$, do not affect the values of $f$ or $g$. Then, for sufficiently large $\beta$ and $\kappa$, we have that
\begin{equation}
    \mathrm{Cov}(f,g) \le C ||f||_{\infty} ||g||_{\infty}   |G|^{8L}|H/H_t|^{8L} \exp\left[ - \mathfrak{c} \frac{\kappa}{6} \right]^{L - |B_1 \cup B_2|}, 
\end{equation}
where $C$ is a constant that does not depend on $|G|,|H/H_t|,\beta$ or $\kappa$, and $\mathfrak{c}$ is a constant defined in Lemma \ref{lem:HiggsProbBoundLargeKappa}.
\end{theorem}

The statement can easily be generalized to the form phrased in Theorem \ref{thm:main-general} with little difficulty.

\subsection{Construction of the Swapping Map}
We will restrict our analysis to the case that $H/H_t$ is the group $Z_2$. We can call one of the charge assignments $+$ and the other charge assignments $-$. Our assignment of Higgs boson fields would be analogous to the assignment of charges in the Ising model.

Let $B_1$ and $B_2$ be two boxes separated by distance at least $L$. 
Let $\calC_1=(\sigma_1,\phi_1)$ and $\calC_2=(\sigma_2,\phi_2)$ be two configurations of gauge and Higgs boson fields.

Let $\mathcal{V}_1 \cup \mathcal{V}_2 \cup \ldots \cup \mathcal{V}_N$ be the vortex decomposition of $\supp_{B_1 \cup B_2}(\calC_1,\calC_2)$. Assume that $B_1$ and $B_2$ are in different vortices of the vortex decomposition. Assume that $B_1 \in \mathcal{V}_1$ and $B_2 \in \mathcal{V}_2$. We will show that we can construct a swapping map in this case.

Observe from the definition of our support that $\sigma_e=1$ for all edges $e$ that are not boundary plaquettes of some plaquette in $\supp(\calC_1,\calC_2)$. It is very easy to define our exchange map for $\sigma$. We let $\Lambda_1(\mathcal{V}_2)$ be the set of edges that form boundary vertices of plaquettes in $\mathcal{V}_2$. We set $\tilde{\sigma}_1$ to be equal to $\sigma_1$ on  $\Lambda_1(\mathcal{V}_2)^c$ and equal to $\sigma_2$ on $\Lambda_1(\mathcal{V}_2)$. We set $\tilde{\sigma}_2$ to be equal to $\sigma_2$ on $\Lambda_1(\mathcal{V}_2)^c$ and equal to $\sigma_1$ on $\Lambda_1(\mathcal{V}_2)$. 

The difficulty is to assign the Higgs boson field charges. Notice that the component of the support due to the assignment of Higgs boson fields is due to the presence of phase boundaries. These phase boundaries will separate regions of $+$ charge from regions of $-$ charge.

\begin{lemma} \label{lemma:chiconst}
Let $\phi$ be an assignment of Higgs boson charges. Let $PB_1\cup \ldots \cup PB_n$ be some collection of phase boundaries found in $\phi$. There exists some map $\chi:\Lambda_0 \to Z_2$ such that $\phi \chi$ will have all the same phase boundaries of $\phi$ except for the union $PB_1 \cup \ldots \cup PB_N$.
\end{lemma}

\begin{proof}
The basic geometry is that one can find islands of different charges with the phase boundaries separating these islands. One can flip the innermost islands and iteratively proceed outwards to get rid of all phase boundaries.

For an example of this construction, consider the case that there is a single phase boundary $PB_1$. This phase boundary surrounds some set of vertices $V$; all of the vertices inside $V$ will have the same Higgs boson charge since there are no other phase boundaries. One can merely flip all of the charges inside $V$ to match the charge outside $V$. 

Now consider the case that we have two phase boundaries, $PB_1$ and $PB_2$. WLOG, one can find two sets of vertices $V_1$ and $V_2$ such that $PB_2$ surrounds $V_2$, $PB_1$ surrounds $V_1$, and $V_2 \subset V_1$. One can assert that all of the charges in $V_2$ are the same while all of the charges of $V_1\setminus V_2$ are the same. To find the map $\chi$, one can first flip the vertices of $V_2$ to match those of $V_1$; this will remove the phase boundary $PB_2$. Afterward, one can flip all of the vertices of $V_1$ to remove the phase boundary $PB_1$. Ultimately, this amounts to fixing the charge assignments of $V_2$ and flipping those of $V_1 \setminus V_2$.

\end{proof}

Let $PB_1^1 \cup PB_2^1 \cup \ldots \cup PB_n^1$ be the phase boundaries of $\calC_1$ found in $\mathcal{V}_1$, and $PB_1^2 \cup PB_2^2 \ldots \cup PB_m^2 $ be the phase boundaries of $\calC_2$ found in $\mathcal{V}_2$. Using Lemma \ref{lemma:chiconst}, one can find some map $\chi_1: \Lambda_0 \to Z_2$ such that the new field $\phi_1 \chi_1$ will have the same phase boundaries as $\phi_1$ outside of the vortex $\mathcal{V}_2$ but will get rid of the phase boundaries $PB_1^1 \cup \ldots PB_n^1$ in $\mathcal{V}_1$.

One can similarly find a map $\chi_2: V_N \to Z_2$ that removes the phase boundaries $PB_1^2\cup \ldots \cup PB_m^2$ and otherwise fixes all other phase boundaries. We define $\tilde{\phi}_1$ to be $\phi_1 \chi_1 \chi_2^{-1}$ and $\tilde{\phi}_1 = \phi_2 \chi_2 \chi_1^{-1}$. We see that $\tilde{\phi}_1$ will have the same phase boundary as $\phi_1$ outside of $\mathcal{V}_2$ and will have the phase boundaries $PB_1^2 \cup \ldots \cup PB_m^2$ inside of $\mathcal{V}_1$. $\tilde{\phi}_2$ will have the phase boundary as $\phi_2$ outside of $\mathcal{V}_2$ and the phase boundary $PB_1^1\cup \ldots \cup PB_n^1$ inside of $\mathcal{V}_2$.

Our swapping map sends $\calC_1,\calC_2 \to (\tilde{\calC}_1,\tilde{\calC}_2)$ where $\calC_1= (\tilde{\sigma}_1,\tilde{\phi}_1)$ and $\calC_2= (\tilde{\sigma}_2,\tilde{\phi}_2).$

\begin{lemma}

Under the construction above we have,
$$
\supp_{B_1 \cup B_2}(\calC_1,\calC_2) = \supp_{B_1 \cup B_2}(\tilde{\calC}_1,\tilde{\calC}_2).
$$

As a consequence of the above, the swapping map constructed above is an involution, and thus, a bijection.
\end{lemma}

\begin{proof}

\textit{Part 1: Equality of the supports}

Let $p$ be a plaquette in $\supp_{B_1 \cup B_2}(\calC_1,\calC_2).$ If $p$ were in $B_1 \cup B_2$, then it would still be in $ \supp_{B_1 \cup B_2}(\tilde{\calC}_1,\tilde{\calC}_2)$. Therefore, we only need to consider the case that $p$ is in the support due to having an excited edge on its boundary. WLOG, assume that $p$ has an excited edge of $\calC_1$.

Assume now that $p$ is in the support due to having an edge with $(\sigma_1)_e \ne 1$. If $e$ were an edge of $\Lambda_1(\mathcal{V}_2)$, then $(\tilde{\sigma}_2)_e = (\sigma_1)_e \ne 1$. Thus, $e$ is an excited edge of $\tilde{\sigma}_2$ and $p$ is in the support of $\mathcal{V}_2$.  If $e \in \Lambda_1(\mathcal{V}_2)^c$, then $(\tilde{\sigma}_1)_e = (\sigma_1)_e$, and $e$ is an excited edge of $\tilde{\calC}_1$. Thus, $p \in \supp(\tilde{\calC}_1).$ 

Now consider the case that $p$ is in the support due to having an edge $e=(x,y)$ in the boundary with $(\phi_1)_x \ne (\phi_1)_y$. This means that the edge $e$ was part of a phase boundary in $\calC_1$. Assume that $p$ is not in $\mathcal{V}_2$. This means that the phase boundary containing the edge $e$ was not one of the phase boundaries $PB_1^1,\ldots, PB_n^1$ that were removed under the transformation $\calC_1 \to \tilde{\calC}_1$.  Thus, $(\tilde{\phi}_1)_x \ne (\tilde{\phi}_1)_y$, and $p$ is in the support of $\tilde{\calC}_1$. If instead, $p$ was in $PB_1^1 \cup \ldots \cup PB_n^1$, then we would have $(\tilde{\phi}_2)_x \ne (\tilde{\phi}_2)_y$, and $p$ would be in the support of $\tilde{\calC}_2$.

This shows that,
$$
\supp_{B_1 \cup B_2}(\calC_1,\calC_2) \subset \supp_{B_1 \cup B_2}(\tilde{\calC}_1,\tilde{\calC}_2).
$$

Now, we show the other inclusion. Let $p$ be a plaquette in  $\supp_{B_1 \cup B_2}(\tilde{\calC}_1,\tilde{\calC}_2)$ and assume that $p$ is not in $B_1 \cup B_2$. WLOG, assume that $p$ is in $\supp(\tilde{\calC}_1)$. Now assume that $p$ is in this support because there is an edge $e \in \delta p$ with $(\tilde{\sigma}_1)_e \ne 1$. If the edge $e$ was in $\Lambda_1(\mathcal{V}_2)$, then $(\sigma_2)_e = (\tilde{\sigma}_1)_e \ne 1$. Thus, $e$ is an excited edge of $\sigma_2$, and $p \in \supp(\tilde{\sigma}_2)$. If instead $e$ was in $\Lambda_1(\mathcal{V}_2)^c$, then $(\sigma_1)_e= (\tilde{\sigma}_1)_e$, and $e$ is an excited edge of $\sigma_1$. Thus, $p \in \supp(\sigma_1)$.

Now consider the case that $p$ has a boundary edge $e=(x,y)$ with $(\tilde{\phi}_1)_x \ne (\tilde{\phi}_1)_y$. Thus, the edge $e$ is part of a phase boundary. The only phase boundaries found in $\tilde{\phi}_1$ are those phase boundaries of $\phi_1$ that lie outside of $\mathcal{V}_2$ (namely, all phase boundaries except for $PB_1^1\cup \ldots \cup PB_n^1$) or one of the phase boundaries $PB_1^2 \cup \ldots \cup PB_m^2$ of $\sigma_2$ that were in $\mathcal{V}_2$. In the former case, $e$ would be part of a phase boundary in $\phi_1$, while it would be part of a phase boundary in $\phi_2$ in the other. In either case, we see that $p$ would be in $\supp_{B_1 \cup B_2}(\calC_1,\calC_2)$. 

This shows the other inclusion and, thus, we have equality of supports. 

\textit{Part 2: Showing the map is an involution}

Now, since the supports are equal, we would have the same vortex decomposition. Thus, the definition of $\mathcal{V}_2$ would be the same whether we consider the pair $(\calC_1,\calC_2)$ or $(\tilde{\calC}_1,\tilde{\calC}_2)$. The exchange map for the gauge field configuration $\sigma$ only depends on knowing the edges of $E(\mathcal{V}_2)$, so it is clearly an involution on these gauge field configurations.

Furthermore, one can see that the relevant phase boundaries of $\tilde{\phi}_1$ are $PB_1^2 \cup \ldots \cup PB_m^2$, while the relevant phase boundaries of $\tilde{\phi}_2$ are $PB_1^1 \cup \ldots \cup PB_n^1$. We see that the map $\chi_2$ would remove the phase boundary of $\tilde{\phi}_1$ while $\chi_1$ would remove the phase boundary of $\tilde{\phi}_2$.  We see that the map applied to $\tilde{\phi}_1$ sends $\tilde{\phi}_1$ to $\tilde{\phi}_1 \chi_2 \chi_1^{-1}= \phi_1 \chi_1 \chi_2^{-1} \chi_2 \chi_1^{-1}= \phi_1$. Similarly, $\tilde{\phi}_2$ would be sent to $\phi_2$. Thus, we see that our proposed map is an involution.
\end{proof}

\begin{lemma}
Let $\mathbb{P}$ be the measure generated by the Hamiltonian $H_{N,\beta,\kappa}$. Then,
\begin{equation}
    \mathbb{P}(\calC_1) \mathbb{P}(\calC_2) = \mathbb{P}(\tilde{\calC_1}) \mathbb{P}(\tilde{\calC_2}).
\end{equation}
\end{lemma}
\begin{proof}
It would suffice to show two equations.

First, we have to show,
\begin{equation}
\begin{aligned}
    &\sum_{p \in \Lambda_2}  [\text{Tr}[\rho((\td \sigma_1)_p)] - \text{Tr}[\rho(1)] ] +  \sum_{p \in \Lambda_2}  [\text{Tr}[\rho((\td \sigma_2)_p)] - \text{Tr}[\rho(1)] ] \\
    & =  \sum_{p \in \Lambda_2} [\text{Tr}[\rho((\td \tilde{\sigma}_1)_p)] - \text{Tr}[\rho(1)] ] +  \sum_{p \in \Lambda_2} [\text{Tr}[\rho((\td \tilde{\sigma}_2)_p)] - \text{Tr}[\rho(1)] ].
\end{aligned}
\end{equation}

The only nontrivial terms on both sides correspond to plaquettes with $(\td \sigma)_p \ne 0$. Let $p$ be a plaquette with $(\td \sigma_1)_p \ne 0$. There exists some edge $e \in \delta p $ with $(\sigma_1)_e \ne 0$. Thus, the plaquette $p$ must belong to $\supp(\sigma_1,\sigma_2)_{B_1 \cup B_2}$. Thus, it suffices to prove the above inequality when restricted to the support $S:=\supp_{B_1 \cup B_2}(\sigma_1,\sigma_2) = \supp_{B_1 \cup B_2}(\tilde{\sigma}_1,\tilde{\sigma}_2)$.

Let $p$ be a plaquette in the support $S$. If all of the boundary edges of $p$ are in $\Lambda_1(\mathcal{V}_2)$, then $(\sigma_1)_e = (\tilde{\sigma}_2)_e$ and $(\sigma_2)_e= (\tilde{\sigma}_1)_e$ for all boundary edges of $p$. This would imply that $\text{Tr}[\rho((\td \sigma_1)_p)]= \text{Tr}[\rho((\td \tilde{\sigma}_2)_p)]$ and $\text{Tr}[\rho((\td \sigma_2)_p)] = \text{Tr}[\rho((\td \tilde{\sigma}_1)_p)]$. We could apply similar logic if all of the boundary edges of $p$ are in $\Lambda_1(\mathcal{V}_2)^c$.

We now need to consider the case that some boundary edges of $p$ are in $\Lambda_1(\mathcal{V}_2)$ while others are in the complement. We first claim that if some boundary edge $e$ of $p$ is in $\Lambda_1(\mathcal{V}_2)$, then $(\sigma_1)_e = (\sigma_2)_e = 1 $. If this were not the case, this would imply that $e$ is an excited edge and, therefore, $p$ must lie in $\mathcal{V}_2$. But, this would imply that all boundary edges of $p$ are in $\Lambda_1(\mathcal{V}_2)$. This contradicts our assumption on $p$. Therefore, $(\sigma_1)_e = (\sigma_2)_e =1$. Furthermore, $(\tilde{\sigma}_1)_e = (\tilde{\sigma}_2)_e =1$. We can use this conclusion to show that $(\sigma_1)_e = (\tilde{\sigma}_1)_e$ and $(\sigma_2)_e= (\tilde{\sigma}_2)_e$ for all edges in the boundary of $p$. As before, this would show that the sums of the Wilson loop actions are the same. This completes the proof of our first equality. 





Secondly, we have to show,
\begin{equation}
\begin{aligned}
    &\sum_{e=(x,y) \in \Lambda_1} (\phi_1)_x \text{Tr}[\rho((\sigma_1)_e)] (\phi_1)_y^{-1} + \sum_{e =(x,y) \in \Lambda_1} (\phi_2)_x \text{Tr}[\rho((\sigma_2)_e)] (\phi_2)_y^{-1}\\
    &= \sum_{e=(x,y) \in \Lambda_1} (\tilde{\phi}_1)_x \text{Tr}[\rho((\tilde{\sigma}_1)_e)] (\tilde{\phi}_1)_y^{-1} + \sum_{e =(x,y) \in \Lambda_1} (\tilde{\phi}_e)_x \text{Tr}[\rho((\tilde{\sigma}_2)_2)] (\tilde{\phi}_2)_y^{-1}
\end{aligned}
\end{equation}

It suffices to prove the above identity as a sum over activated edges. Let $e$ be an activated edge. If $e$ were in  $\Lambda_1(\mathcal{V}_2)$, this would mean that $(\sigma_1)_e= (\tilde{\sigma}_2)_e$ and $(\sigma_2)_e = (\tilde{\sigma}_1)_e$. Furthermore,  $(\phi_1)_x (\phi_1)_y^{-1}= (\tilde{\phi}_2)_x (\tilde{\phi}_2)_y^{-1}$ and $(\phi_2)_x (\phi_2)_y^{-1}= (\tilde{\phi}_1)_x (\tilde{\phi}_1)_y^{-1}$. This is due to the fact if $e$ were part of a phase boundary in $\Lambda_1(\mathcal{V}_2)$, then we would perform a flipping so that the phase boundaries of $\phi_1$ match those of $\tilde{\phi}_2$ while the phase boundaries of $\phi_2$ match those of $\tilde{\phi}_1$ in $\mathcal{V}_2$. Otherwise, $e$ is not part of a phase boundary, and all quantities are trivially equal to $1$.

Ultimately, this means that, $(\phi_1)_x \text{Tr}[\rho((\sigma_1)_e)] (\phi_1)_y^{-1} = (\tilde{\phi}_2)_x \text{Tr}[\rho((\tilde{\sigma}_2)_e)] (\tilde{\phi}_2)_y^{-1}$ and $(\phi_2))x \text{Tr}[\rho((\sigma_2)_e)] (\phi_2)_y^{-1} = (\tilde{\phi}_1)_x \text{Tr}[\rho((\tilde{\sigma}_1)_e)] (\tilde{\phi}_1)_y^{-1}$.

If, instead, $e$ was in $\Lambda_1(\mathcal{V}_2)$, we would know that $(\phi_1)_x \text{Tr}[\rho((\sigma_1)_e)] (\phi_1)_y^{-1} = (\tilde{\phi}_1)_x \text{Tr}[\rho((\tilde{\sigma}_1)_e)] (\tilde{\phi}_1)_y^{-1}$  and $(\phi_2)_x \text{Tr}[\rho((\sigma_2)_e)] (\phi_2)_y^{-1} = (\tilde{\phi}_2)_x \text{Tr}[\rho((\tilde{\sigma}_2)_e)] (\tilde{\phi}_2)_y^{-1}$. We can sum this relationship over all edges to get the desired energy equality. 

\end{proof}

\subsection{Completing the Argument}

The discussion of the previous section establishes that we have a swapping map. At this point, we can apply Lemma \ref{lemma:swapping-gives-covariance-bound} to derive our decay of correlation bounds. The only possibility we need to exclude is that there is a vortex that would connect boxes $B_1$ and $B_2$ that are separated by distance $L$. The probability of this event occurring can be followed by using the polymer counting functions of Section \ref{section:probability-bound}. We remark that our bounds are even easier since the condition for splitting vortices is only determined by the condition of compatibility, i.e., whether they are adjacent to each other or not. We give the following definition:
\begin{defn}
Our polymer counting function $\Phi^{(2)}_{P_0}$ is defined as follows:
\begin{equation}
    \Phi^{(2)}_{P_0}(P) = \sum_{\substack{\supp_{P_0}(\calC_1=(\sigma_1,\phi_1),\calC_2=(\sigma_2,\phi_2))=P\\ \sum_{v \in \Lambda_0}(\phi_1)_v>0, \sum_{v \in \Lambda_0} \sum_{v \in \Lambda_0} (\phi_2)_v >0}} \exp[H_{N,\beta,\kappa}(\calC_1)] \exp[H_{N,\beta,\kappa}(\calC_2)]
\end{equation}
\end{defn}

This polymer counting function satisfies properties similar to those outlined in Lemmas \ref{lemma:Phi-2-factorization} and \ref{lemma:knot-appear-prob-bound-by-Phi}. Just as these lemmas are consequences of splittings applied to each individual component (either $\psi_1$ or $\psi_2$) independently, we can apply the arguments of \cite[Lemma 2]{Adh2021} to $\calC_1$ and $\calC_2$ separately to derive these splitting lemmas( for vortices rather than knots). The only lemma we would need to change is the quantitative probability bound derived in Lemma \ref{lemma:Phi-bound}; to bound the total probability percolation, one can sum quantities of the form $\Phi^{(2)}_{B_1 \cup B_2}(V)$ over appropriate vortices $V$ using an argument similar to the proof of  Proposition \ref{prop:percolation-probability-bound}.

In the next lemma, we derive an analog of the bound in Lemma \ref{lemma:Phi-bound}. An argument similar to the proof of Proposition \ref{prop:percolation-probability-bound} to bound the probability of percolation using estimates on $\Phi^{(2)}_{B_1 \cup B_2}(P)$ is left to the reader.

\begin{lemma} \label{lem:HiggsProbBoundLargeKappa}
Let $P$ be a plaquette set that cannot be decomposed into disjoint vortices. Assume that $P$ contains the boxes $B_1$ and $B_2$ that are separated by distance $K$.
Furthermore, define the constant $\mathfrak{c}$ as $$ \mathfrak{c}:=2\max_{(a,b) \ne(1,1) } \text{Re}[\text{Tr}[a \rho(b)] - \text{Tr}[\rho(1)]], $$ where $a$ can vary over entries in $H/H_t$, and $b$ can vary over entries in $G$.
Then, we have the following bound on $\Phi^{(2)}_{B_1 \cup B_2}$.
\begin{equation}
    \Phi^2_{B_1 \cup B_2} \le 4^{|P|} |G|^{8|P|} |H/H_t|^{8|P|} \exp[- \mathfrak{c}\frac{\kappa}{6}]^{|P|- |B_1 \cup B_2|}
\end{equation}

\end{lemma}
\begin{proof}
    Consider two configurations $\calC_1$ and $\calC_2$ such that $\supp_{B_1 \cup B_2}(\calC_1,\calC_2)= P$. Then, $\supp(\calC_1)$ and $\supp(\calC_2)$ must be subsets of $P$. There are thus $4^{|P|}$ ways to choose the supports $\supp(\calC_1)$ and $\supp(\calC_2)$. Furthermore, the only edges that have a nontrivial gauge field are those on the boundary of the plaquettes of $P$. There are at most $4|P|$ of these boundary edges and $|G|^{8|P|}$ ways to assign the $\sigma$'s for $\calC_1$ or $\calC_2$. Similarly, the only nontrivial Higgs boson values with $\phi_v \ne 1$ are those found on boundary vertices of $P$. This will give us $|H/H_t|^{8|P|}$ ways to assign $\phi_1$ and $\phi_2$.
    
    Finally, we remark that each plaquette $p \in P \setminus (B_1 \cup B_2)$ must be excited by having an excited edge $e$ on the boundary. Furthermore, an excited edge can excite at most 6 plaquettes. Therefore, we must have at least $\frac{(|P| - |B_1 \cup B_2)}{6}$ excited edges. Each of these excited edges will contribute $\exp[-\mathfrak{c} \kappa]$.
\end{proof}

\section{Higgs Boson Small \texorpdfstring{$\kappa$}{k}}\label{appendix:higgs-small-kappa}

The previously discussed paper \cite{Adh2021} also considered the case when $\kappa$ is small. In contrast to the case when $\kappa$ is large, there is no longer any compulsion for $\phi_x= \phi_y$.  Clusters are no longer defined by looking at the boundaries between distinct $\phi_x$.

Instead, the main idea of \cite{Adh2021} was to expand the part of the exponential containing $ \kappa \sum_{e=(x,y) \in \Lambda_1^u} \left[2 \text{Re}[\phi_x \text{Tr}[\rho(\sigma_e)] \phi_y^{-1}] +c \right]$ in the Hamiltonian using the power series $\exp[x] = \sum_{i=0}^{\infty} \frac{x^i}{i!}$. This introduces the new random variable $i$, representing the power used in the expansion along each edge. The swapping argument used in this section will not involve swapping the values of the Higgs field $\phi_x$. The method, instead, will involve swapping the values of the field $I$, while summing over the Higgs field values $\phi$. We will give the details in what follows.

\subsection{Introduction to the model}
To deal with the small $\kappa$ case, we first have to expand the Hamiltonian via the random current expansion.

Let $c$ be a constant such that,
$$
2 \text{Re}[\phi_x \text{Tr}[\rho(\sigma_e)] \phi_y^{-1}] +c > 0,
$$
for all values of $\phi_x,\phi_y$ and $\sigma_e$.

We can consider the Hamiltonian,
\begin{equation}
    H_{N,\beta,\kappa}^1(\sigma,\phi) = \beta \sum_{e \in \Lambda_1} \text{Tr}[\rho(\sigma_e)- \rho(1)] + \kappa \sum_{e=(x,y) \in \Lambda_1^u} \left[2 \text{Re}[\phi_x \text{Tr}[\rho(\sigma_e)] \phi_y^{-1}] +c \right],
\end{equation}
where the sum $e\in \Lambda_1^u$ is over the set of unoriented edges. Namely, instead of oriented pairs $e$ and $-e$, we only include a single unoriented edge.

The random current expansion of this Hamiltonian would be,
\begin{equation}
\begin{aligned}
    \mathcal{H}_{N,\beta,\kappa}(\sigma,\phi,I) &= \beta \sum_{e \in \Lambda_1} \text{Tr}[\rho(\sigma_e) - \rho(1)] \\
     & + \kappa \sum_{e = (x,y) \in \Lambda_1^u} I(e) \log \left[2 \text{Re}[\phi_x \text{Tr}[\rho(\sigma_e)] \phi_y^{-1}] +c \right] - \log I(e)!.
\end{aligned}
\end{equation}

We have a new field $I(e)$ that takes non-negative integer values. Marginalizing over the $I(e)$ variables would return our original Hamiltonian $H_{N,\beta,\kappa}^1$. We have a new definition of support according to our new Hamiltonian $\mathcal{H}_{N,\beta,\kappa}$.

\begin{defn}
We define our set of activated edges $AE$ to be those edges $e$ with $I(e) \ne 0$.
\begin{equation}
    AE:=\{e: I(e) \ne 0 \},
\end{equation}
and the set of activated vertices $AV$ is the set,
\begin{equation}
    AV:=\{v: \exists e \in AE \text{  s.t. }  v\in \delta e \}.
\end{equation}

We can now define the support of our configuration $ \calC =(\sigma,\phi,I)$ to consist of those plaquettes which have a boundary vertex in $AV$ or those plaquettes with $(\td \sigma)_p \ne 1$.

\begin{equation}
    \supp(\calC) =\{ p \in P_N : \exists v \in AV \text{ s.t. } v \in \delta\delta p \text{ or } (\td \sigma)_p \ne 1 \},
\end{equation}

Given two configurations $\calC_1$ and $\calC_2$, we define 
\begin{equation}
    \supp_P(\calC_1,\calC_2) = \supp(\calC_1) \cup \supp(\calC_2) \cup P.
\end{equation}
\end{defn}

Instead of directly swapping configurations of our Hamiltonian, we consider a swapping model on our associated reduced Gibbs' measure. We can parameterize our configurations on our reduced Gibbs' measure as
$(\psi,\phi,I)$, where $\psi \in \text{Hom}(\pi_1(S_1(\Lambda),x_0),G)$. For each $\psi$, let $\Sigma$ denote the set of gauge field configurations $\sigma$ that map to $\psi$.

Our reduced Gibbs' measure has the following distribution,
\begin{equation}
    \mathcal{G}(\psi,\phi,I) = \sum_{\sigma \in \Sigma} \frac{1}{Z}\exp[\mathcal{H}_{N,\beta,\kappa}(\sigma,\phi,I)].
\end{equation}

Our goal in this section is to prove the following theorem,
\begin{theorem}
Consider two boxes $B_1$, $B_2$ separated by $\ell^{\infty}$ distance $L$. Additionally,  $f$ and $g$ be functions such that $f$ only depends on the values of $\sigma$ and $\phi$ in the box $B_1$, while $g$ depends only on the values of $\sigma$ and $\phi$ in the box $B_2$. Furthermore, assume that transformations of the form $\sigma_e \to \eta_x \sigma_e \eta_y^{-1}$ for some auxiliary field $\eta:\Lambda_0 \to G$, do not affect the values of $f$ or $g$. Then, for sufficiently large $\beta$ and $\kappa$, we have that
\begin{equation}
    \mathrm{Cov}(f,g) \le  C ||f||_{\infty} ||g||_{\infty}  
 |G|^{2L}   \mathfrak{c}^{L - |B_1 \cup B_2|}, 
\end{equation}
where $C$ is a constant that does not depend on $|G|,|H|,\beta$ or $\kappa$, and $\mathfrak{c}$ is a constant defined in Lemma \ref{lem:probBoundHiggssmallkappa}.
\end{theorem}

Our gauge invariance condition includes examples of actions such as Wilson loop expectations. Furthermore, the gauge invariance condition means that we would only need to consider functions of $\psi$ rather than $\sigma$ and, therefore, we could consider the auxiliary Hamiltonian $\mathcal{H}$ and reduced Gibbs' measure $\mathcal{G}$. In the remainder of this section, we construct 
We prove this theorem by defining a swapping map for this reduced Gibbs' measure.

\subsection{Swapping on reduced configurations}

This section will construct an almost swapping for the gauge class.

Namely, we will do the following. Let $R\calC_1=(\psi_1,\phi_1,I_1)$ and $R\calC_2=(\psi_2,\phi_2,I_2)$ be two reduced configurations.
We denote by $\Sigma_1$ the set of all $\sigma$'s that will map to the homomorphism $\psi_1$ and by $\Sigma_2$ the set of $\sigma$'s that map to the homomorphism $\psi_2$. 

Let $B_1$ and $B_2$ be two boxes separated by distance $L$.
Let $K_1 \cup \ldots \cup K_n$ be the knot decomposition corresponding to $\supp_{B_1 \cup B_2} (\calC_1,\calC_2)$. Observe that if we replace $\sigma_1$ with any other gauge field assignment in $\Sigma_1$, we would have the same knot decomposition. Assume that $B_1$ is in $K_i$ and $B_2$ is in $K_j$ for $i \ne j.$

Apply the mapping $T:(\psi_1,\psi_2) \to (\tilde{\psi}_1,\tilde{\psi}_2)$ according to the knot decomposition $K_1 \cup \ldots \cup K_n$ as in Section \ref{section:construction-of-swapping-map}. Let $\tilde{\Sigma}_1$ and $\tilde{\Sigma}_2$ be the sets of configurations that correspond to $\tilde{\psi}_1$ and $\tilde{\psi}_2$. In addition to this, switch $I_1(e)$, $I_2(e)$ on the set of activated edges of $K_j$ and switch $\phi_1$, $\phi_2$ on the activated vertices of $K_j$; these switches will give us $\tilde{I}_1,\tilde{I}_2$ and $\tilde{\phi}_1,\tilde{\phi}_2$. We denote the full map from $(\psi_1,\phi_1,I_1),(\psi_2,\phi_2,I_2) \to (\tilde{\psi}_1,\tilde{\phi}_1,\tilde{I}_1),(\tilde{\psi}_2, \tilde{\phi}_2,\tilde{I}_2)$ as $\hat{T}$. We will first show that $\hat{T}$ is an involution.
\begin{lemma}
Consider two configurations $R\calC_1=(\psi_1,\phi_1,I_1)$ and $R\calC_2=(\psi_2,\phi_2,I_2)$. We first have that,
\begin{equation}
    \supp_{B_1 \cup B_2}(R\calC_1,R\calC_2) = \supp_{B_1 \cup B_2}(\hat{T}(R\calC_1,R\calC_2)|_1, \hat{T}(R\calC_1,R\calC_2)|_2).
\end{equation}

As a consequence of the above statement on the supports, we can argue that $\hat{T}$ is an involution and, thus, a bijection.

\end{lemma}

\begin{proof}
Let $p$ be a plaquette in the support $\supp_{B_1 \cup B_2}(R\calC_1,R\calC_2)$ that is not part of $B_1 \cup B_2$. This support has the knot decomposition $K_1 \cup \ldots K_n$ where $K_i$ contains $B_1$, and $K_j$ contains $B_2$ with $K_j \ne K_i$.  There are two possibilities for $p$: either there is a boundary vertex of $p$ that is simultaneously a boundary vertex for some activated edge in $AE_1$ (the activated edges of $R\calC_1$) or $AE_2$ (the activated edges of $R\calC_2$), or we have a nontrivial current around a vertex with $\psi_1(\xi_p) \ne 1$ or $\psi_1(\xi_p)\ne 1$.  

Now, consider the first case that $p$ shares a boundary vertex with an edge of $AE_1$ or $AE_2$. WLOG, we may assume that this is an edge $e$ of $AE_1$. If this plaquette were part of the knot $K_j$, then $e$ would be an activated edge of the configuration $\hat{T}(R\calC_1,R\calC_2)|_2$ by our definition of the map $\hat{T}$. Thus, $p$ would be part of the support of $\hat{T}(R\calC_1,R\calC_2)|_2$ and also in 
$\supp_{B_1 \cup B_2}(\hat{T}(R\calC_1,R\calC_2)|_1,\hat{T}(R\calC_1,R\calC_2)|_2)$. If, instead, the plaquette $p$ was not part of the knot $K_j$, then $e$ would instead be an activated edge of the configuration $\hat{T}(R\calC_1,R\calC_2)|_1$, and we could apply the same logic.

If we consider the case that $p$ has nontrivial current $\psi_1(\xi_p)\ne 1$ or $\psi_2(\xi_p) \ne 1$, then the fact that $p \in \supp_{B_1\cup B_2}(\hat{T}(R\calC_1,R\calC_2)|_1, \hat{T}(R\calC_1,R\calC_2)|_2)$ is true from the corresponding property of $T$ in Lemma \ref{lemma:Tisinvolution}.

This shows that $$\supp_{B_1 \cup B_2}(R\calC_1,R\calC_2) \subset \supp_{B_1 \cup B_2}(\hat{T}(R\calC_1,R\calC_2)|_1,\hat{T}(R\calC_1,R\calC_2)|_2)$$.

The other inclusion is similar. If $p$ were in $$\supp_{B_1 \cup B_2}(\hat{T}(R\calC_1,R\calC_2)|_1,\hat{T}(R\calC_1,R\calC_2)|_2),$$ then either $p$ shares a boundary vertex with an activated edge in one of $\hat{T}(R\calC_1,R\calC_2)|_1$ or $\hat{T}(R\calC_1,R\calC_2)|_2$, or it has a nontrivial current with some $\psi(\xi_p) \ne 1$.

The case that $p$ has a nontrivial current was again treated in Lemma \ref{lemma:Tisinvolution}. Now we consider the case that $p$ shares a boundary vertex with an activated edge. WLOG, assume the activated edge was in $\hat{T}(R\calC_1,R\calC_2)|_1$. All activated edges of $\hat{T}(R\calC_1,R\calC_2)$ come from an activated edge in $R\calC_1$ ( if $p$ was not part of $K_j$) or an activated edge in $R\calC_2$ ( if $p$ was part of $K_j$). Thus, $p$ still shares a boundary vertex of an activated edge either in $R\calC_1$ or in $R\calC_2$. In either case, $p$ would belong to $\supp_{B_1 \cup B_2}(R\calC_1,R\calC_2)$ as desired.

This shows the other inclusion,
$$
 \supp_{B_1 \cup B_2}(\hat{T}(R\calC_1,R\calC_2)|_1,\hat{T}(R\calC_1,R\calC_2)|_2) \subset \supp_{B_1 \cup B_2}(R\calC_1,R\calC_2).
$$

Since the supports of both sides are equal, we see that both supports would have the same knot decomposition $K_1\cup \ldots \cup K_n$. By the definition of $\hat{T}$, $\hat{T}$ applied to $\hat{T}(R\calC_1,R\calC_2)$ would reverse the switching of the activated edges in $K_j$ we performed in the first $\hat{T}$. Lemma \ref{lemma:Tisinvolution} additionally shows that $\hat{T}$ will be an involution when restricted to the coordinates $(\psi_1,\psi_2)$ (since $\hat{T}$ is equal to $T$ on this pair).

\end{proof}


The remainder of this section is devoted to showing that $\hat{T}$ preserves probability.

It will be a consequence of the following lemma.
\begin{lemma} \label{lemma:decompRC}
Let $R\calC=(\psi,\phi,I)$ have a support with knot decomposition $K_1 \cup K_2 \ldots \cup K_n$. Let $B_1$ be the box that separates $K_1$ from $K_2 \cup \ldots \cup K_n$.  Let $\Sigma$ be the set of all configurations $\sigma$ that map to $\psi$.

Now, we can define notions that are related to our splitting.
Let $\phi^{in}$, $I^{in}$ denote the values of $\phi$ and $I$ restricted to the activated vertices and edges in $K_1$ (so $\phi^{in}$ and $I^{in}$ take trivial values for all other vertices and edges). We similarly let $\phi^{out}$ and $I^{out}$ denote the values of $\phi$ and $I$ not included in $\phi^{in}$ and $I^{in}$ (so they would take trivial values on the activated vertices and edges in $K_1$). Finally, let $(\psi^{in},\psi^{out})$ denote the splitting of the homomorphism $\psi$ in accordance with the separation by the box $B_1$ in Lemma  \ref{lemma:bijection-existence}.
Let $\Sigma^{in}$ ($\Sigma^{out}$) be the set of all gauge configurations $\sigma$ that map to the homomorphism $\psi^{in}$ ($\psi^{out}$).

Then, we have that,
\begin{equation}
\begin{aligned}
    \frac{1}{|G|^{|\Lambda_0|-1}} \sum_{\sigma \in \Sigma} \exp[\mathcal{H}_{N,\beta,\kappa}(\sigma,\phi,I)] &= \frac{1}{|G|^{|\Lambda_0|-1}} \sum_{\sigma \in \Sigma^{in}}\exp[\mathcal{H}_{N,\beta,\kappa}(\sigma,\phi^{in},I^{in})]\\& \times \frac{1}{|G|^{|\Lambda_0|-1}}\sum_{\sigma \in \Sigma^{out}}\exp[\mathcal{H}_{N,\beta,\kappa}(\sigma,\phi^{out},I^{out})] .
\end{aligned}
\end{equation}

One can iterate this construction to derive components $\phi^{in,j}$ ,$I^{in,j}$, $\Sigma^{in,j}$ at the iteration of splitting $K_j$ from $K_{j+1} \cup \ldots K_{n}$. ( With $\phi^{in,1}=\phi^{in}$).

As a consequence, we notice that,
\begin{equation}
    \frac{1}{|G|^{|\Lambda_0|-1}} \sum_{\sigma \in \Sigma} \exp[\mathcal{H}_{N,\beta,\kappa}(\sigma,\phi,I)] = \prod_{j=1}^n \frac{1}{|G|^{|\Lambda_0|-1}} \sum_{\sigma \in \Sigma^{j}} \exp[\mathcal{H}_{N,\beta,\kappa}(\sigma,\phi^{in,j}, I^{in,j})]
\end{equation}

\end{lemma}
\begin{proof}
With respect to $B_1$, one can choose a representative $\sigma$ in $\Sigma$ such that $\sigma$ can decompose as a product $\sigma^{in} \sigma^{out}$. $\sigma^{in}$ satisfies the property that its only nontrivial edge values are inside the box $B_1$, and $\sigma^{out}$ has its only nontrivial edges outside the box $B_1$.

By introducing an auxiliary field $\eta$, we see that we have,
\begin{equation}
\begin{aligned}
    &\frac{1}{|G|^{|\Lambda_0|-1}} \sum_{\sigma \in \Sigma} \exp[\mathcal{H}_{N,\beta,\kappa}(\sigma,\phi,I)]\\
    &= \prod_{p \in \Lambda_2} \varphi_{\beta}(\psi(\xi_p)) \frac{1}{|G|^{|\Lambda_0|}} \sum_{\eta: \Lambda_0 \to G}\\
    &\times \exp \left[  \sum_{e=(x,y):I(e) \ne 0} I(e) \log[2 \text{Re}[(\phi_x) \text{Tr}[\rho(\eta_x) \rho((\sigma_1)_e) \rho((\sigma_2)_e) \rho(\eta_y^{-1})](\phi_y)^{-1}] + c] \right]\\
    &= \prod_{p \in \Lambda_2(K_1)} \varphi_{\beta}(\psi^{in}(\xi_p)) \frac{1}{|G|^{|\Lambda_0(K_1)|}} \sum_{\eta^{in}: \Lambda_0(K_1) \to G}\\
    &\times \exp\left[ \sum_{e=(x,y):I^{in}(e) \ne 0} I^{in}(e)\log[ 2 \text{Re}[\phi^{in}_x \text{Tr}[\rho(\eta^{in}_x) \rho((\sigma_1)_e) \rho((\eta^{in}_y)^{-1})] (\phi^{in}_y)^{-1}] +c ]  \right]
    \\& \times \prod_{p \in \Lambda_2(K_1)} \varphi_{\beta}(\psi^{out}(\xi_p)) \frac{1}{|G|^{|\Lambda_0(K_1)|^c}} \sum_{\eta^{out}:\Lambda_0(K_1)^c \to G}\\
    & \times \exp \left[ \sum_{e=(x,y):I^{out} \ne 0} I^{out}(e) \log[ 2 \text{Re}[\phi^{out}_x \text{Tr}[\rho(\eta_x^{out}) \rho((\sigma_2)_e) \rho((\eta_y^{out})^{-1})] (\phi^{out}_y)^{-1}] + c] \right]
\end{aligned}
\end{equation}

The first equality merely reparameterizes the configurations in $\sigma$ in terms of the auxiliary field $\eta$. The point of the last two lines is that $B_1$ can split the connected clusters of $I^{in}$ and $I^{out}$ of the box $B_1$. This causes the Higgs boson interaction to split so that the Higgs boson interaction inside the box only depends on $I^{in}$, $\phi^{in}$, $\sigma_1$, and $\eta^{in}$. 

One can create dummy summation variables to turn $\eta^{in}$ and $\eta^{out}$ into full auxiliary fields rather than restricted auxiliary fields. Once this is done and a compensating power of $\frac{1}{|G|}$ is added, one sees that the last two lines are a representation of $\frac{1}{|G|^{|\Lambda_0|-1}} \sum_{\sigma \in \Sigma^{out}} \exp[\mathcal{H}_{N,\beta,\kappa}(\sigma,\phi^{out},I^{out})]$, and the two lines preceding those represent  $\frac{1}{|G|^{|\Lambda_0|-1}} \sum_{\sigma \in \Sigma^{in}} \exp[\mathcal{H}_{N,\beta,\kappa}(\sigma,\phi^{in},I^{in})]$.

\end{proof}

An immediate corollary of the above fact is our desired multiplicative identity.

\begin{cor}
Let $R\calC_1=(\psi_1,\phi_1,I_1)$ and $R\calC_2=(\psi_2,\phi_2,I_2)$ be two configurations with $\supp_{B_1\cup B_2}(R\calC_1,R\calC_2) = K_1 \cup \ldots \cup K_n$ with $B_1$ in $K_i$ and $B_2$ in $K_j$ with $j \ne 1$. Let the configurations after switching be $R\tilde{\calC}_1 = (\tilde{\psi}_1,\tilde{\phi}_1,\tilde{I}_1)$ and $R\tilde{\calC}_2= (\tilde{\psi}_2,\tilde{\phi}_2,\tilde{I}_2)$. After the switching, we have the probability conservation relation:
\begin{equation}
\begin{aligned}
    & \frac{1}{|G|^{2|\Lambda_0|-2}}\sum_{\sigma \in \Sigma_1} \sum_{\sigma \in \Sigma_2}\exp[ \mathcal{H}_{N,\beta,\kappa}(\sigma,\phi_1,I_1)] \times  \exp[\mathcal{H}_{N,\beta,\kappa}(\sigma,\phi_2,I_2)]=\\
    & \frac{1}{|G|^{2|\Lambda_0|-2}}\sum_{\sigma \in \tilde{\Sigma}_1} \sum_{\sigma \in \Sigma_2} \exp[\mathcal{H}_{N,\beta,\kappa}(\sigma,\tilde{\phi}_1,\tilde{I}_1)] \times \exp[\mathcal{H}_{N,\beta,\kappa}(\sigma,\tilde{\phi}_2,\tilde{I}_2)]. 
\end{aligned}
\end{equation}
\end{cor}

\begin{proof}
By using Lemma \ref{lemma:decompRC}, we see that we can decompose each sum of the form $\frac{1}{|G|^{|\Lambda_0|-1}} \sum_{\sigma \in \Sigma_1} \exp[\mathcal{H}_{N,\beta,\kappa}(\sigma,\phi_1,I_1)]$ into its constituent parts $\psi_1^{l,in},\phi_1^{l,in}$, and $ I_1^{l,in}$ with respect to the division given by the knot decomposition $K_1 \cup K_2 \cup \ldots \cup K_n$ and write it as a product.

\begin{equation}
\begin{aligned}
    &\frac{1}{|G|^{|\Lambda_0|-1}} \sum_{\sigma \in \Sigma_1} \exp[\mathcal{H}_{N,\beta,\kappa}(\sigma,\phi_1,I_1)] = \prod_{l=1}^n \frac{1}{|G|^{|\Lambda_0|-1}} \sum_{\sigma \in \Sigma_1^{l,in}} \exp[\mathcal{H}_{N,\beta,\kappa}(\sigma,\phi_1^{l,in}, I_1^{l,in})],\\
    &\frac{1}{|G|^{|\Lambda_0|-1}} \sum_{\sigma \in \Sigma_2} \exp[\mathcal{H}_{N,\beta,\kappa}(\sigma,\phi_2,I_2)] = \prod_{l=1}^n \frac{1}{|G|^{|\Lambda_0|-1}} \sum_{\sigma \in \Sigma_2^{l,in}} \exp[\mathcal{H}_{N,\beta,\kappa}(\sigma,\phi_2^{l,in}, I_2^{l,in})].
\end{aligned}
\end{equation}

Now, the configurations $(\tilde{\psi}_1,\tilde{\phi}_1, \tilde{I}_1)$ and $(\tilde{\psi}_1, \tilde{\phi}_2,\tilde{I}_2)$ are formed by exchanging the components $(\psi_1^{j,in},\phi_1^{j,in},I_1^{j,in})$ and $(\psi_2^{j,in},\phi_2^{j,in},I_2^{j,in})$. Thus, we have that,
\begin{equation}
\begin{aligned}
    \frac{1}{|G|^{|\Lambda_0|-1}} \sum_{\sigma \in \tilde{\Sigma}_1} \exp[\mathcal{H}_{N,\beta,\kappa}(\sigma,\tilde{\phi}_1,\tilde{I}_1)] &= \prod_{l\ne j} \frac{1}{|G|^{|\Lambda_0|-1}} \sum_{\sigma \in \Sigma_1^{l,in}} \exp[\mathcal{H}_{N,\beta,\kappa}(\sigma,\phi_1^{l,in}, I_1^{l,in})]\\
    & \times \frac{1}{|G|^{|\Lambda_0|-1}} \sum_{\sigma \in \Sigma_2^{j,in}} \exp[\mathcal{H}_{N,\beta,\kappa}(\sigma,\phi_2^{j,in}, I_2^{j,in})].
\end{aligned}
\end{equation}

A similar expression holds for $(\tilde{\psi}_2,\tilde{\phi}_2,\tilde{I}_2)$. From these explicit decompositions, one can see that the products of the expressions are the same. This shows that the swapping $\hat{T}$ preserves probabilities.
\end{proof}

A consequence of the preceding corollary is that $\hat{T}$ is a swapping map on reduced configurations. 

\subsection{Probability Bounds}

As we have mentioned in the last section, we have constructed a swapping map that allows one to switch reduced configurations as long as that, in the knot decomposition, the boxes $B_1$ and $B_2$ are in different knots in the support. Thus, we have a decay of correlation bound whose error probability is bounded by the probability that $B_1$ and $B_2$ belong in the same knot in the knot decomposition. In this section, we define a polymer counting function that allows us to bound the probability of this event.

\begin{defn}
Out polymer counting function $\Phi^{(2)}_{P_0}$ is defined as follows:
\begin{equation}
    \Phi^{(2)}_{P_0}(P)= \frac{1}{(|G||H/H_t|)^{2|\Lambda_0|-2}} \sum_{\supp_{P_0}(\calC_1,\calC_2)=P} \exp[\mathcal{H}_{N,\beta,\kappa}(\calC_1)] \exp[\mathcal{H}_{N,\beta,\kappa}( \calC_2)].
\end{equation}
\end{defn}

The main property we need regarding the above polymer counting function is that it splits as a product when there are two sets $P_1$ and $P_2$ in the support that are separated by a box $B$. This is an analog of Lemma 
\ref{lemma:Phi-2-factorization} and can be proved using the methods of \cite[Lemma 9]{Adh2021}. A consequence of this splitting is that the probability that there is a knot $K$ that contains both boxes $B_1$ and $B_2$ is bounded by $\Phi^{(2)}_{B_1 \cup B_2}(K)$. Once we have a quantitative bound on $\Phi^{(2)}_{B_1 \cup B_2}(K)$ based on $\exp[-c|K|]$ for some $c$, we can  bound the contribution of all knots by using the arguments in the proof of Proposition \ref{prop:percolation-probability-bound}. Our next lemma gives a bound on $\Phi^{(2)}_{B_1 \cup B_2}(K)$, and we leave the full details of the bounds coming from summing up knots to the reader.

\begin{lemma} \label{lem:probBoundHiggssmallkappa}
We have the following bound on $\Phi^{(2)}_{P_0}(P)$.

We define $\mathfrak{c}$ as,
\begin{equation}
    \mathfrak{c}:= \max\left(\exp[\beta(max_{g \ne 1} 2 \text{Re}[\text{Tr}[\rho(g) - \rho(1)]])], \frac{\kappa}{24} \exp[\max_{a \ne b }(2 \text{Re}[a\text{Tr}[\rho(b)]] + c)]\right)
\end{equation}
\begin{equation} \label{eq:phi2bndhiggs}
    \Phi^{(2)}_{P_0}(P) \le 16^{|P|} G^{2|P|} 2^{8|P|} \mathfrak{c}^{|P|-|P_0|}.
\end{equation}
\end{lemma}

\begin{proof}
Consider two configurations $\calC_1$ and $\calC_2$ with joint support $P$. This means that the supports of $\calC_1$ and $\calC_2$ are subsets of $P$. There are $4^|P|$ ways to choose the support of $\calC_1$ and $\calC_2$. Furthermore, the set of plaquettes $p$ with nontrivial circulation ($\td \sigma \ne 1$) in $\calC_1$ is a subset of the support of $\calC_1$ ( and similar with $\calC_2$). Thus, there are at most $4^|P|$ ways to choose the set of plaquettes $P_1$( resp. $P_2$) with nontrivial circulation in $\calC_1$ (resp. $\calC_2$). This gives us our first factor of $16^{|P|}$ in Equation \eqref{eq:phi2bndhiggs}.

Given the set of plaquettes $P_1$ with nontrivial circulation, there are at most $|G|^{|P_1|}$ homomorphisms $\psi_1$ that will have $P_1$ as the subset with nontrivial circulation. Similarly, there are at most $|G|^{|P_2|}$ homomorphisms $\psi_2$ with $P_2$ as the subset with nontrivial circulation. The factor of $|G|^{-(|\Lambda_0|-1)}$ will cancel out the factor coming from considering the gauge configurations $\sigma$ that would map to a homomorphism $\psi_1$ or $\psi_2$. This explains the factor $|G|^{2|P|}$ in Equation \eqref{eq:phi2bndhiggs}.

Furthermore, the only activated edges in our configuration must be boundary edges of plaquettes in $P$. There are at most $2^{4|P|}$ ways to choose the set of activated edges in $\calC_1$ and $2^{4|P|}$ ways to choose the set of activated edges in $\calC_2$. This explains the third factor $2^{8|P|}$ in Equation \eqref{eq:phi2bndhiggs}.

Finally, we remark that if a plaquette is in $P \setminus P_0$, then either the plaquette has nontrivial circulation in one of $\calC_1$ or $\calC_2$ or there is an activated edge in $\calC_1$ or $\calC_2$ that shares a boundary vertex with $p$. The first quantity in the minimum appearing in $\mathfrak{c}$ bounds the contribution from a plaquette with nontrivial circulation. The second factor bounds the contribution from the existence of an activated edge. The bound comes from summing over $I(e)$ from $1$ to $\infty$ and further noticing that a single activated edge could result in at most $24$ plaquettes appearing in the support. This gives us our last factor of $\mathfrak{c}^{|P|-|P_0|}$.
\end{proof}

\bibliographystyle{plainnat}

\begin{thebibliography}{99}

\bibitem{Adh2021}
    {\sc Adhikari,A.}(2021)
    {Wilson loop expectations for non-abelian gauge fields coupled to a Higgs boson at low and high disorder.}
    Preprint available at
    \href{https://arxiv.org/abs/2111.07540}{arXiv:2111.07540}.
    
    
\bibitem{AF1984}
    {\sc Aizenman, M. and Fr\"{o}hlich, J.} (1984).
    {Topological anomalies in the $n$ dependence of the $n$-states Potts lattice gauge theory.}
    {\em Nuclear Phys. B}, {\bf 235} no. 1, 1-18. 
    
\bibitem{AP2019}
    {\sc Aizenman, M. and Peled, R.} (2019).
    {A power-law upper bound on the correlations in the 2D random field Ising model.}
    {\em Comm. Math. Phys.} {\bf 372}, 865-892.
    
\bibitem{AHP2020}
    {\sc Aizenman, M., Harel, M. and Peled, R.} (2020).
    {Exponential decay of correlations in the 2D random field Ising model.}
    {\em J. Stat. Phys.}, {\bf 180}, 304-331.

\bibitem{Aizenman82}
    {\sc Aizenman,M.} (1982),
    {Geometric Analysis of $\phi^4$ Fields and Ising Models. Parts I and II.}
    {\em Comm. Math. Phys.}, {\bf 86}, 1-48.
    
\bibitem{balaban83} {\sc Ba\l aban, T.} (1983). Regularity and decay of lattice Green's functions. {\it Comm. Math. Phys.,} {\bf 89} no. 4, 571--597.

\bibitem{balaban84a} {\sc Ba\l aban, T.} (1984a). Renormalization group methods in non-abelian gauge theories. {\it Harvard preprint, HUTMP B134.}

\bibitem{balaban84b} {\sc Ba\l aban, T.} (1984b). Propagators and renormalization transformations for lattice gauge theories. I. {\it Comm. Math. Phys.,} {\bf 95} no. 1, 17--40.

\bibitem{balaban84c} {\sc Ba\l aban, T.} (1984c). Propagators and renormalization transformations for lattice gauge theories. II. {\it Comm. Math. Phys.,} {\bf 96} no. 2, 223--250.

\bibitem{balaban84d} {\sc Ba\l aban, T.} (1984d). Recent results in constructing gauge fields. {\it Physica A,} {\bf 124} no. 1-3, 79--90.

\bibitem{balaban85a} {\sc Ba\l aban, T.} (1985a). Averaging operations for lattice gauge theories. {\it Comm. Math. Phys.,} {\bf 98} no. 1, 17--51.

\bibitem{balaban85b} {\sc Ba\l aban, T.} (1985b). Spaces of regular gauge field configurations on a lattice and gauge fixing conditions.
{\it Comm. Math. Phys.,} {\bf 99} no. 1, 75--102.

\bibitem{balaban85c} {\sc Ba\l aban, T.} (1985c). Propagators for lattice gauge theories in a background field. {\it Comm. Math. Phys.,} {\bf 99} no. 3, 389--434.

\bibitem{balaban85d} {\sc Ba\l aban, T.} (1985d). Ultraviolet stability of three-dimensional lattice pure gauge field theories. {\it Comm. Math. Phys.,} {\bf 102} no. 2, 255--275. 

\bibitem{balaban85e} {\sc Ba\l aban, T.} (1985e). The variational problem and background fields in renormalization group method for lattice gauge theories. {\it Comm. Math. Phys.,} {\bf 102} no. 2, 277--309.

\bibitem{balaban87} {\sc Ba\l aban, T.} (1987). Renormalization group approach to lattice gauge field theories. I. Generation of effective actions in a small field approximation and a coupling constant renormalization in four dimensions. {\it Comm. Math. Phys.,} {\bf 109} no. 2, 249--301.

\bibitem{balaban88} {\sc Ba\l aban, T.} (1988). Convergent renormalization expansions for lattice gauge theories. {\it Comm. Math. Phys.,} {\bf 119} no. 2, 243--285.

\bibitem{balaban89a} {\sc Ba\l aban, T.} (1989a). Large field renormalization. I. The basic step of the {\bf R} operation. {\it Comm. Math. Phys.,} {\bf 122} no. 2, 175--202. 

\bibitem{balaban89b} {\sc Ba\l aban, T.} (1989b). Large field renormalization. II. Localization, exponentiation, and bounds for the {\bf R} operation. {\it Comm. Math. Phys.,} {\bf 122} no. 3, 355--392.

\bibitem{BS16} {\sc Basu, R.} and {\sc Ganguly, S.} (2016). $SO(N)$ Lattice Gauge Theory, planar and beyond. To appear in {\it Comm. Pure Appl. Math.}

\bibitem{BFP2010}
    {\sc Bissacot, R., Fern\'{a}ndez, R. and Procacci, A.} (2010).
    {On the convergence of cluster expansions for polymer gases.}
    {\em J. Stat. Phys.}, {\bf 139} 598-617.
    
\bibitem{Borgs1984}
    {\sc Borgs, C.} (1984).
    {Translation symmetry breaking in four-dimensional lattice gauge theories.}
    {\em Comm. Math. Phys.}, {\bf 96} no. 2, 251-284.
    
\bibitem{borgs88} {\sc Borgs, C.} (1988). Confinement, deconfinement and freezing in lattice Yang-Mills theories with continuous time. 
{\it Comm. Math. Phys.,} {\bf 116} no. 2, 309--342. 

\bibitem{Cao2020}
    {\sc Cao, S.} (2020).
    {Wilson loop expectations in lattice gauge theories with finite gauge groups.}
    {\em Comm. Math. Phys.}, {\bf 380}, 1439-1505.
    
\bibitem{Ch15} {\sc Chatterjee, S.} (2015). Rigorous solution of strongly coupled $SO(N)$ lattice gauge theory in the large $N$ limit. To appear in {\it Comm. Math. Phys.}

\bibitem{Ch16} {\sc Chatterjee, S.} (2016). The leading term of the Yang-Mills free energy. {\it J. Funct. Anal.,} {\bf 271}, 2944--3005.


\bibitem{CJ16} {\sc Chatterjee, S.} and {\sc Jafarov, J.} (2016). The $1/N$ expansion for $SO(N)$ lattice gauge theory at strong coupling. {\it Preprint.} Available at arXiv:1604.04777.
    
\bibitem{Ch2018}
    {\sc Chatterjee, S.} (2018).
    {Yang-Mills for probabilists.}
    {\em Probability and Analysis in Interacting Particle Systems - in honor of S.R.S. Varadhan.}
    
\bibitem{Ch2021}
    {\sc Chatterjee, S.} (2021).
    {A probabilistic mechanism for quark confinement.} {\em Comm. Math. Phys.}, {\bf 385}, 1007-1039.
    
\bibitem{CJR1979} 
    {\sc Creutz, M., Jacobs, L. and Rebbi, C.} (1979).
    {Monte carlo study of abelian lattice gauge theories.}
    {\em Phys. Rev. D.}, {\bf 20} no. 8, 1915-1922. 
    
\bibitem{FOS04} {\sc P. A. Faria da Veiga, O’Carroll, M. and Schor, R.} (2004).
Existence of baryons, baryon spectrum and mass splitting in strong
coupling lattice QCD. Comm. Math. Phys., 245 no. 2, 383–405.   
   
 
\bibitem{DX2021}
    {\sc Ding, J. and Xia, J.} (2021).
    {Exponential decay of correlations in the two-dimensional random field Ising model at zero temperature.} 
    {\em Inventiones Mathematicae.} {\bf 224}, 999-1045.


\bibitem{DPSS2017}
{\sc Duminil-Copin, H., Peled, R., Samotij, W., and Spinka, Y.} (2017)
{Exponential decay of loop lengths in the loop O(n) model.} 
{\em Comm. Math. Phys.}, {\bf 349}, 777-817.
    
\bibitem{Forsstrom2021a}
    {\sc Forsstr\"{o}m, M.P.} (2021).
    {Decay of correlations in finite Abelian lattice gauge theories.} 
    Preprint available at
    \href{https://arxiv.org/abs/2104.03752}{arXiv:2104.03752}.
    
\bibitem{Forsstrom2021b}
    {\sc Forsstr\"{o}m, M.P.} (2021).
    {Wilson lines in the Abelian lattice Higgs model.} 
    Preprint available at
    \href{https://arxiv.org/abs/2111.06620}{arXiv:2111.06620}.
    
\bibitem{FLV2020}
    {\sc Forsstr\"{o}m, M.P., Lenells, J., and Viklund, F.} (2020).
    {Wilson loops in finite Abelian lattice gauge theories.}
    To appear in {\em Ann. de l'inst. Henri Poincar\'{e} Probab. Stat.} 
    
\bibitem{FLV2021}
    {\sc Forsstr\"{o}m, M.P., Lenells, J., and Viklund, F.} (2021).
    {Wilson loops in the abelian lattice Higgs model.}
    Preprint available at
    \href{https://arxiv.org/abs/2111.06620}{arXiv:2111.06620}.
    
\bibitem{Fro1979}
    {\sc Fr\"ohlich, J.} (1979).
    {Confinement in $Z_n$ lattice gauge theories implies confinement in $SU(n)$ lattice Higgs theories.}
    {\em Phys. Lett. B.} {\bf 83} no. 2, 195-198.
    
\bibitem{frohlichspencer82} {\sc Fr\"ohlich, J.} and {\sc Spencer, T.} (1982). Massless phases and symmetry restoration in abelian gauge theories and spin systems. {\it Comm. Math. Phys.,} {\bf 83} no. 3, 411--454.
    
\bibitem{GS2021}
    {\sc Garban, C. and Sep\'{u}lveda, A.} (2021).
    {Improved spin-wave estimate for Wilson loops in $U(1)$ lattice gauge theory.}
    \href{https://arxiv.org/abs/2107.04021}{arXiv:2107.04021}.
\bibitem{GlJa}{\sc Glimm, J. and Jaffe, A.} (1987). Quantum physics. A functional integral point of view. Second edition. Springer-Verlag, New York.    
    
    
\bibitem{gopfertmack82} {\sc G\"opfert, M.} and {\sc Mack, G.} (1982). Proof of confinement of static quarks in 3-dimensional $U(1)$ lattice gauge theory for all values of the coupling constant. {\it Comm. Math. Phys.,} {\bf 82} no. 4, 545--606.
    
\bibitem{guth80} {\sc Guth, A.~H.} (1980). Existence proof of a nonconfining phase in four-dimensional $U(1)$ lattice gauge theory. {\it Phys. Rev. D,} {\bf 21} no. 8, 2291--2307.

\bibitem{J16} {\sc Jafarov, J.} (2016). Wilson loop expectations in $SU(N)$ lattice gauge theory. {\it Preprint.} Available at arXiv:1610.03821.

\bibitem{LT2021}
    {\sc Lees, B. and Taggi, L.} (2021).
    {Exponential decay of transverse correlations for O(N) spin systems and related models.}
    {\em Prob. Theory. Relat. Fields}, {\bf 180}, 1099-1133.

    
\bibitem{MP1979}
    {\sc Mack, G. and Petkova, V.B.} (1979).
    {Comparison of lattice gauge theories with gauge groups $\Z_2$ and $SU(2)$.}
    {\em Ann. Physics}, {\bf 123} no. 2, 442-467.
    
    
\bibitem{os78} {\sc Osterwalder, K.} and {\sc Seiler, E.} (1978). Gauge field theories on a lattice. {\it Ann. Physics,} {\bf 110} no. 2, 440--471.

\bibitem{PL78}{ \sc Paes-Leme, P. J.} (1978). Ornstein-Zernike and analyticity properties
for classical lattice spin systems. Ann. Physics, 115 no. 2, 367–387.

    
\bibitem{Seiler1982}
    {\sc Seiler, E.} (1982).
    {\em Gauge theories as a problem of constructive field theroy and statistical mechanics.} Springer-Verlag, Berlin.
    
\bibitem{SV1989}
    {\sc Szlach\`{a}nyi, K. and Vecserny\`{e}s P.} (1989).
    {Cluster expansion in terms of knots in gauge theories with finite non-Abelian gauge groups}
    {\em J. Math. Phys.}, {\bf 30} no. 9, 2156-2159.
    
\bibitem{STILL1993}
    {\sc Stillwell, J.} (1993).
    {\em Classical topology and combinatorial group theory.} Graduate Texts in Math. {\bf 72}, Springer-Verlag 1982; Second edition 1993.
    
\bibitem{Tom1993}
    {\sc Tomboulis, E.T.} (1993).
    {Confinement via dynamical monopoles.}
    {\em Phy. Lett. B.} {\bf 303} no. 1-2, 103-108.
    
\bibitem{WEG1971}
    {\sc Wegner, F.J.} (1971).
    {Duality in generalized Ising models and phase transitions without local order parameters.}
    {\em J. Math. Phys.}, {\bf 12} no. 10, 2259-2272.
    
\bibitem{W1974} {\sc Wilson, K.~G.} (1974). Confinement of quarks. {\it Phys. Rev. D,} {\bf 10} no. 8, 2445--2459.




\end{thebibliography}


\end{document}